\documentclass{amsart}
\usepackage[T1]{fontenc}
\usepackage{lmodern}
\usepackage[latin1]{inputenc}
\usepackage{amsmath}
\usepackage{amssymb}
\usepackage{amsthm}
\usepackage[enableskew]{youngtab}
\usepackage{hyperref}
\newcommand{\arxiv}[1]{\href{http://arxiv.org/abs/#1}{\texttt{arXiv:#1}}}
\usepackage{breakurl}
\usepackage{color}
\usepackage{paralist}
\usepackage{setspace}
\usepackage{graphicx}
\newcommand{\NN}{\ensuremath{\mathbb N}}

\theoremstyle{plain}
\newtheorem{Th}{Theorem}[section]
\newtheorem{Lem}[Th]{Lemma}
\newtheorem{Co}[Th]{Corollary}
\newtheorem{Prop}[Th]{Proposition}
\newtheorem{Hyp}[Th]{Hypothesis}
\theoremstyle{definition}
\newtheorem{Def}[Th]{Definition}
\newtheorem{Ex}[Th]{Example}
\newtheorem*{Ack}{Acknowledgement}
\theoremstyle{remark}
\newtheorem{Rem}[Th]{Remark}
\newtheorem*{Not}{\underline{\bf{Notation}}}
\newtheorem*{Con}{\underline{\bf{Convention}}}
\numberwithin{equation}{section}
\newcommand{\meins}{1'}
\newcommand{\mzwei}{2'}
\newcommand{\mdrei}{3'}
\newcommand{\mvier}{4'}
\newcommand{\mfuenf}{5'}
\newcommand{\msechs}{6'}

\newcommand{\deins}{\mathbf 1}
\newcommand{\dzwei}{\mathbf 2}

\newcommand{\dmeins}{\mathbf 1'}
\newcommand{\dmzwei}{\mathbf 2'}

\newcommand{\none}{\textcolor{white}{.}}

\begin{document}
\title{Classification of $Q$-multiplicity-free skew Schur $Q$-functions}
\author{Christopher Schure}
\address{\scshape Institut für Algebra, Zahlentheorie und Diskrete Mathematik, Leibniz Universität Hannover, Welfengarten 1, D-30167 Hannover; schure@math.uni-hannover.de}
\begin{abstract}
We classify the $Q$-multiplicity-free skew Schur $Q$-functions.
Towards this result, we also provide new relations between the shifted Littlewood-Richardson coefficients.
\end{abstract}

\maketitle

\section{Introduction}
Schur functions form an important basis of the algebra of symmetric functions.
They appear in the study of the representations of the symmetric groups and the general linear groups.
Schur $P$-functions and Schur $Q$-functions are bases of the subalgebra generated by the odd power sums.
In \cite{Stembridge}, Stembridge proved a number of important properties of Schur $Q$-functions emphasizing that they may be viewed as shifted analogues of Schur functions.
While in the classical situation Schur functions are closely related to ordinary irreducible characters of the symmetric groups,
Schur $Q$-functions are intimately connected to irreducible spin characters of the double covers of the symmetric groups.
Multiplicity-free products of Schur functions were classified by Stembridge in \cite{Stembridge3}.
As a shifted analogue of Stembridge's result, Bessenrodt then classified the $P$-multiplicity-free products of Schur $P$-functions in \cite{Bessenrodt}.
A skew generalization of Stembridge's result was proved independently by Gutschwager in \cite{Gutschwager} and Thomas and Yong in \cite{ThomasYong}.
While Gutschwager classified the multiplicity-free skew Schur functions, Thomas and Yong classified the multiplicity-free products of Schubert classes.
However, what was missing was a skew analogue of Bessenrodt's result or equivalently a shifted analogue of Gutschwager's result.
The main goal of this article is to provide this here, i.e., to classify the $Q$-multiplicity-free skew Schur $Q$-functions.
We will heavily rely on the shifted Littlewood-Richardson rule obtained by Stembridge in \cite{Stembridge} (another version of this rule was given by Cho \cite{Cho}).\par
The paper is structured as follows.
In the second section we provide the required definitions and some properties needed later.
In the third section we prove relations between shifted Littlewood-Richardson coefficients, which will simplify proofs of the fourth section.
In the fourth section we will first exclude all non-$Q$-multiplicity-free skew Schur $Q$-functions before proving the $Q$-multiplicity-freeness of the remaining skew Schur $Q$-functions to obtain our main classification result, Theorem \ref{listmf}.
Note that we define some special notation for partitions with distinct parts whose shifted diagrams have at most two corners in Definition \ref{shapepath}.
We will use that notation in most lemmas of the fourth section.

\section{Preliminaries}
We will use the same notation as in \cite{Schure}.
Some of the tools introduced there will also be useful in the context here.

\subsection{Partitions, diagrams and tableaux}
We define a \textbf{partition} as a tuple $\lambda = (\lambda_1, \lambda_2, \ldots, \lambda_n)$ where $\lambda_j \in \NN$ for all $1 \leq j \leq n$ and $\lambda_i \geq \lambda_{i+1} > 0$ for all $1 \leq i \leq n-1$.
The \textbf{length} of $\lambda$ is $\ell(\lambda) := n$.
A partition $\lambda$ is called a partition of $k$ if $|\lambda| := \lambda_1 + \lambda_2 + \ldots + \lambda_{\ell(\lambda)} = k$ where $|\lambda|$ is called the \textbf{size} of $\lambda$.
A \textbf{partition with distinct parts} is a partition $\lambda = (\lambda_1, \lambda_2, \ldots, \lambda_n)$ where $\lambda_i > \lambda_{i+1} > 0$ for all $1 \leq i \leq n-1$.
The set of partitions of $k$ with distinct parts is denoted by $DP_k$.
By definition, the empty partition $\emptyset$ is the only element in $DP_0$ and it has length $0$.
The \textbf{set of all partitions with distinct parts} is denoted by $DP := \bigcup_k DP_k$.
For $\lambda \in DP$ the \textbf{shifted diagram} $D_{\lambda}$ is defined by $D_{\lambda} := \{(i,j) \mid 1 \leq i \leq \ell(\lambda), i \leq j \leq i+\lambda_i-1\}$.
\begin{Con}
In this paper we will omit the adjective shifted.
This means that whenever a diagram is mentioned it is always a shifted diagram.
\end{Con}
For $\lambda, \mu \in DP$ with $\ell(\mu) \leq \ell(\lambda)$ and $\mu_i \leq \lambda_i$ for all $1 \leq i \leq \ell(\mu)$, we define the \textbf{skew diagram} $D_{\lambda/\mu} := D_{\lambda} \setminus D_{\mu}$.
Its \textbf{size} is $|D_{\lambda/\mu}| = |D_\lambda| - |D_\mu|$.\par
Each edgewise connected part of a skew diagram $D$
is called a \textbf{component}.
The number of components of $D$ is denoted by $comp(D)$.
If $comp(D) = 1$, then $D$ is called \textbf{connected}, otherwise it is called \textbf{disconnected}.\par
In the following, if components are numbered, the numbering is as follows: the first component is the leftmost component, the second component is the next component to the right of the first component etc.\par
A \textbf{corner} of a skew diagram $D$ is a box $(x,y) \in D$ such that $(x+1,y), (x,y+1) \notin D$.
An \textbf{unshifted} diagram is a skew diagram $D_{\lambda/\mu}$ with $\ell(\mu) = \ell(\lambda)-1$.
A (skew) diagram can be depicted as an arrangement of boxes where the coordinates of the boxes are interpreted in matrix notation.
\begin{Ex}
Let $\lambda = (6,5,2,1)$ and $\mu = (4,3)$.
Then the (skew) diagram is
$$D_{\lambda/\mu} = {\Yvcentermath1 \young(::\none \none ,::\none \times ,\none \none ,:\times ).}$$
We have $|D_{\lambda/\mu}| = 7$.
The diagram $D_{\lambda/\mu}$ has two components where the first component consists of three boxes and the second component consists of four boxes.
The corners of $D_{\lambda/\mu}$ are the boxes marked $\times$.
\end{Ex}

We consider the alphabet $\mathcal{A} = \{1' < 1 < 2' < 2 < \ldots\}$.
The letters $1, 2, 3, \ldots$ are called \textbf{unmarked} letters and the letters denoted by $1', 2', 3', \ldots$ are called \textbf{marked} letters.
For a letter $x$ of the alphabet $\mathcal{A}$ we denote the unmarked version of this letter with $|x|$.

\begin{Def}\label{tableaudef}
Let $\lambda, \mu \in DP$ with skew diagram  $D_{\lambda/\mu}$.
A \textbf{tableau} $T$ of shape $D_{\lambda/\mu}$ is a map $T: D_{\lambda/\mu} \rightarrow \mathcal{A}$ such that
\begin{enumerate}[a)]
	\item $T(i,j) \leq T(i+1,j)$, $T(i,j) \leq T(i,j+1)$ for all $i,j$,
	\item each column has at most one $k$ ($k = 1, 2, 3, \ldots$),
	\item each row has at most one $k'$ ($k' = 1', 2', 3', \ldots$).
\end{enumerate}
Let $c^{(u)}(T) = (c^{(u)}_1, c^{(u)}_2, \ldots)$ where $c^{(u)}_i$ denotes the number of letters equal to $i$ in the tableau $T$, for each $i$.
Analogously, let $c^{(m)}(T) = (c^{(m)}_1, c^{(m)}_2, \ldots)$ where $c^{(m)}_i$ denotes the number of $i'$s in the tableau $T$, for each $i$.
Then the \textbf{content} of $T$ is defined by $c(T) = (c_1, c_2, \ldots) := c^{(u)}(T) + c^{(m)}(T)$.
If there is some $k$ such that $c_k > 0$ but $c_j = 0$ for all $j > k$ then we omit all these $c_j$ from $c(T)$.
\end{Def}
\begin{Rem}
We depict a tableau $T$ of shape $D_{\lambda/\mu}$ by filling each box $(x,y)$ of the diagram with the letter $T(x,y)$.
\end{Rem}
\begin{Ex}
Let $\lambda = (8,6,5,3,2)$ and $\mu = (5,2,1)$.
Then a tableau of shape $D_{\lambda/\mu}$ is
$$T = {\Yvcentermath1 \young(::\meins 12,\mzwei 224,2455,4\msechs 6,:67).}$$
We have $c(T) = (2,5,0,3,2,3,1)$.
\end{Ex}

\subsection{Skew Schur $Q$-functions}
For $\lambda, \mu \in DP$ and a countable set of independent variables $x_1, x_2, \ldots$ the \textbf{skew Schur $Q$-function} is defined by
$$Q_{\lambda/\mu} := \sum_{T \in T(\lambda/\mu)}{x^{c(T)}}$$
where $T(\lambda/\mu)$ denotes the set of all tableaux of shape $D_{\lambda/\mu}$ and $x^{(c_1, c_2, \ldots, c_{\ell})} := x_1^{c_1} x_2^{c_2} \cdots$ with $c_k := 0$ for $k > \ell$.
If $D_{\mu} \nsubseteq D_{\lambda}$ then $Q_{\lambda/\mu} := 0$.
Since $D_{\lambda/\emptyset} = D_{\lambda}$, we denote $Q_{\lambda/\emptyset}$ by $Q_{\lambda}$.

\begin{Def}
Let a diagram $D$ be such that the $y^{\textrm{th}}$ column has no box, but there are boxes to the right of the $y^{\textrm{th}}$ column and after shifting all boxes that are to the right of the $y^{\textrm{th}}$ column one box to the left we obtain a diagram $D_{\alpha/\beta}$ for some $\alpha, \beta \in DP$.
Then we call the $y^{\textrm{th}}$ column \textbf{empty} and the diagram $D_{\alpha/\beta}$ is obtained by removing the $y^{\textrm{th}}$ column.
Similarly, let a diagram $D$ be such that the $x^{\textrm{th}}$ row has no box,  but there are boxes below the $x^{\textrm{th}}$ row and after shifting all boxes that are below the $x^{\textrm{th}}$ row one box up and then all boxes of the diagram one box to the left, we obtain a diagram $D_{\alpha/\beta}$ for some $\alpha, \beta \in DP$.
Then we call the $x^{\textrm{th}}$ row \textbf{empty} and the diagram $D_{\alpha/\beta}$ is obtained by removing the $x^{\textrm{th}}$ row.
\end{Def}

\begin{Def}
For $\lambda, \mu \in DP$ we call the diagram $D_{\lambda/\mu}$ \textbf{basic} if it satisfies the following properties:
\begin{itemize}
	\item $D_{\mu} \subseteq D_{\lambda}$,
	\item $\ell(\lambda) > \ell(\mu)$,
	\item $\lambda_i > \mu_i$,  for all $1 \leq i \leq \ell(\mu)$,
	\item $\lambda_{i+1} \geq \mu_i-1$,  for all $1 \leq i \leq \ell(\mu)$.
\end{itemize}
This means that $D_{\lambda/\mu}$ has no empty rows or columns.
\end{Def}

For a given diagram $D$ let $\bar{D}$ be the diagram obtained by removing all empty rows and columns of the diagram $D$.
Since the tableau restrictions on each box in a diagram are unaffected by removing empty rows and columns, there is a content-preserving bijection between tableaux of a given shape and tableaux of the shape obtained by removing empty rows and columns; thus we have $Q_D = Q_{\bar{D}}$.
Hence, we may restrict our considerations to partitions $\lambda$ and $\mu$ such that $D_{\lambda/\mu}$ is basic.\par
For some given skew diagram $D$ let the diagram obtained after removing empty rows and columns be $D_{\lambda/\mu}$ for some $\lambda, \mu \in DP$.
Then $Q_D$ is equal to the skew Schur $Q$-function $Q_{\lambda/\mu}$.

For a tableau $T$ of shape $D$, the \textbf{reading word} $w = w(T)$ is the word obtained by reading the rows from left to right beginning with the bottom row and ending with the top row.
The \textbf{length} $\ell(w)$ is the number of letters and, thus, the number
$n=|D|$ of boxes in $D$.
Let $(x(i),y(i))$ denote the box of the $i^{\textrm{th}}$ letter of the reading word $w(T)$.\par
For the reading word $w=w_1w_2\ldots w_n$
of the tableau $T$ the statistics $m_i(j)$ are defined as follows:
\begin{itemize}
	\item $m_i(0) = 0$ for all $i$.
	\item For $1 \leq j \leq n$ the number $m_i(j)$ is equal to the number of times $i$ occurs in the word $w_{n-j+1} \dots w_n$.
	\item For $n+1 \leq j \leq 2n$ we set $m_i(j) := m_i(n) + k(i)$ where $k(i)$ is the number of times $i'$ occurs in the word $w_1 \dots w_{j-n}$.
\end{itemize}
As Stembridge remarked \cite[before Theorem 8.3]{Stembridge}, the statistics $m_i(j)$ for some given $i$ can be calculated  by taking the word $w(T)$ and scan it first from right to left while counting the letters $i$ and afterwards scan it from left to right and adding the number of letters $i'$.
After the $j^{\textrm{th}}$ step of scanning and counting the statistic $m_i(j)$ is calculated.\begin{Def}\label{amenable}
Let $k \in \NN$ and $w=w_1w_2\ldots w_n$ be a word of length $n$ consisting of letters from the alphabet $\mathcal{A}$.
The word $w$ is called $k$\textbf{-amenable} if it satisfies the following conditions:
\begin{enumerate}[a)]
	\item if $m_k(j) = m_{k-1}(j)$ then $w_{n-j} \notin \{k, k'\}$ for all $0 \leq j \leq n-1$,
	\item if $m_k(j) = m_{k-1}(j)$ then $w_{j-n+1} \notin \{k-1, k'\}$ for all $n \leq j \leq 2n-1$,
	\item if $j$ is the smallest number such that $w_j \in \{k', k\}$ then $w_j = k$,
	\item if $j$ is the smallest number such that $w_j \in \{(k-1)', k-1\}$ then $w_j = k-1$.
\end{enumerate}
Note that $c^{(u)}_i = m_i(n)$ and $c^{(m)}_i = m_i(2n) - m_i(n)$.\par
The word $w$ is called \textbf{amenable} if it is $k$-amenable for all $k > 1$.
A tableau $T$ is called ($k$-)amenable if $w(T)$ is ($k$-)amenable.
\end{Def}
\begin{Lem}\cite[Lemma 3.28]{Salmasian}\label{unmarked}
Let $w$ be a $k$-amenable word for some $k \geq 1$.
Let $n = \ell(w)$.
If $m_{k-1}(n) > 0$ then $m_{k-1}(n) > m_k(n)$.
\end{Lem}

\begin{Lem}\label{firstrows}
Let $T$ be an amenable tableau.
Then there are no entries greater than~$k$ in the first $k$ rows.
\end{Lem}
\begin{proof}
Assume the opposite.
Let $i$ be the topmost row with an entry greater than~$i$.
Let this entry be $x$.
Then, while scanning the word $w(T)$ from right to left, the letter $x$ will be scanned before any letter $|x|-1$ will be scanned; a contradiction to the amenability of the tableau $T$.
\end{proof}
Using Lemma \ref{firstrows}, one can construct a brute force algorithm to obtain amenable tableaux of a given shape.
The algorithm fills the first row with elements from $\{1', 1\}$ satisfying the conditions of Definition \ref{tableaudef}.
Each box of the $i^{\textrm{th}}$ row gets filled with entries at most $i$ and greater or equal the entry of the box above (if there is a box above) and greater or equal the entry of the box to the left (if there is a box to the left), again in a way such that the $i^{\textrm{th}}$ row satisfies Definition \ref{tableaudef}.
After all boxes are filled the algorithm takes the reading word and uses Stembridge's scanning algorithm to check if this filling is amenable.

For the proofs of the lemmas in Section~4 we will use the following shifted Littlewood-Richardson rule by Stembridge.

\begin{Th}\cite[Theorem 8.3]{Stembridge}
For $\lambda, \mu \in DP$ we have
$$Q_{\lambda/\mu} = \sum_{\nu \in DP}{f^{\lambda}_{\mu \nu}} Q_{\nu},$$
where $f^{\lambda}_{\mu \nu}$ is the number of amenable tableaux $T$ of shape $D_{\lambda/\mu}$ and content $\nu$.
\end{Th}

For $\lambda \in DP$, the corresponding Schur $P$-function is defined by $P_{\lambda} := 2^{-\ell(\lambda)} Q_{\lambda}$.
In  \cite[Chapter 8]{Stembridge}, Stembridge showed that the numbers $f^{\lambda}_{\mu \nu}$ above also appear in the product of $P$-functions:
$$P_{\mu} P_{\nu} = \sum_{\lambda \in DP}{f^{\lambda}_{\mu \nu} P_{\lambda}}\:.$$
Using this, one easily obtains the equation $f^\lambda_{\mu\nu} = f^\lambda_{\nu\mu}$ for all $\lambda, \mu, \nu \in DP$.\par

\begin{Def}
A \textbf{border strip} is a connected (skew) diagram $B$ such that for each $(x,y) \in B$ we have $(x-1,y-1) \notin B$.
The box $(x,y) \in B$ such that $(x-1,y) \notin B$ and $(x,y+1) \notin B$ is called the \textbf{first box} of $B$.
The box $(u,v) \in B$ such that $(u+1,v) \notin B$ and $(u,v-1) \notin B$ is called the \textbf{last box} of $B$.\par
A (possibly disconnected) diagram $D$ where all components are border strips is called a \textbf{broken border strip}.
Then the first box of the rightmost component is called the \textbf{first box of $D$}, and the last box of the leftmost component is called the \textbf{last box of $D$}.\par
A $(p,q)$\textbf{-hook} is a set of boxes
$$\{(u,v+q-1), \ldots, (u,v+1), (u,v), (u+1, v), \ldots, (u+p-1, v)\}$$
for some $u, v \in \NN$.
More precisely, we say that the set of boxes above is a $(p,q)$-hook at $(u,v)$.
\end{Def}

\begin{Def}\label{boxeswithi}
Let $T$ be a tableau of shape $D_{\lambda/\mu}$.
Define $T^{(i)}$ by
$$T^{(i)} := \{(x,y) \in D_{\lambda/\mu} \mid |T(x,y)| = i\}.$$
\end{Def}

\begin{Lem}\label{diagonal}\cite[Remark before Theorem 13.1]{HoffmanHumphreys}
Let $T$ be a tableau of shape $D_{\lambda/\mu}$.
Then $|T(x,y)| < |T(x+1,y+1)|$ for all $x,y$ such that $(x,y), (x+1,y+1) \in D_{\lambda/\mu}$.
\end{Lem}
As a consequence of this lemma, for a given tableau $T$ each component of $T^{(i)}$ is a border strip.
This fact as well as the following lemma are derived from \cite[after Corollary 8.6]{SaganStanley}.
\begin{Lem}\label{2fillings}
Let $T$ be a tableau of shape $D_{\lambda/\mu}$.
Let $T^{(i)}$ be of shape $D$ for some diagram $D$.
Then each component of $D$ has two possible fillings which differ only in the last box of this component.
\end{Lem}

We will use a criterion for $k$-amenability of a tableau that avoids the use of the reading word.
This is provided in Lemma \ref{checklist}; to state this lemma we need the following definitions.

\begin{Def}\label{fitting}
Let $T$ be a tableau.
If the last box of $T^{(i)}$ is filled with $i$ we call $T^{(i)}$ \textbf{fitting}.
\end{Def}

\begin{Def}
Let $\lambda, \mu \in DP$ and let $T$ be a tableau of $D_{\lambda/\mu}$.
Then
$$\mathcal{S}^{\boxtimes}_{\lambda/\mu}(x,y) := \{(u,v) \in D_{\lambda/\mu} \mid u \leq x, v \geq y\},$$
$$\mathcal{S}^{\boxtimes}_T(x,y)^{(i)} := \mathcal{S}^{\boxtimes}_{\lambda/\mu}(x,y) \cap T^{-1}(i) \text{ where $T^{-1}(i)$ denotes the preimage of $i$},$$
$$\mathcal{B}_T^{(i)}:=\{(x,y) \in D_{\lambda/\mu} \mid T(x,y) = i' \text{ and } T(x-1,y-1) \neq (i-1)'\},$$
$$\widehat{\mathcal{B}_T^{(i)}}:=\{(x,y) \in D_{\lambda/\mu} \mid T(x,y) = i' \text{ and } T(x+1,y+1) \neq (i+1)'\}$$
and $b_T^{(i)} = |\mathcal{B}_T^{(i)}|$ for all $i$.
Then let $\mathcal{B}_T^{(i)}(d)$ denote the set of the first $d$ boxes of $\mathcal{B}_T^{(i)}$.
\end{Def}

\begin{Rem}
The set $\mathcal{S}^{\boxtimes}_{\lambda/\mu}(x,y)$ above is the set of boxes that are simultaneously weakly above and weakly to the right of the box $(x,y)$.
The set $\mathcal{S}^{\boxtimes}_T(x,y)^{(i)}$ is the subset of boxes $(u,v)$ of $\mathcal{S}^{\boxtimes}_{\lambda/\mu}(x,y)$ such that $T(u,v) = i$.
\end{Rem}

\begin{Ex}
Let $\lambda = (11,9,6,5,4,2,1)$ and $\mu = (8,6,5,4,1)$ and let
$${\Yvcentermath1 T = \young(\times \times \times \times \times \times \times \times \dmeins \deins \deins ,:\times \times \times \times \times \times \dmeins \dmzwei \dzwei ,::\times \times \times \times \times \deins ,:::\times \times \times \times \mzwei ,::::\times \meins 12,:::::1\mzwei ,::::::2)}.$$
Then $\mathcal{S}^{\boxtimes}_{\lambda/\mu}(3,8)$ is the set of boxes with boldfaced entries.
Also, we have $\mathcal{S}^{\boxtimes}_T(3,8)^{(1)} = \{(1,10), (1,11), (3,8)\}$, $\mathcal{B}_T^{(2)} = \{(2,9), (4,8)\}$ and $\widehat{\mathcal{B}_T^{(1)}} = \{(1,9), (2,8)\}$.
\end{Ex}

\begin{Lem}\cite[Lemma 2.14]{Schure}\label{checklist}
Let $\lambda, \mu \in DP$ and $n = |D_{\lambda/\mu}|$.
Let $T$ be a tableau of $D_{\lambda/\mu}$.
Then the tableau $T$ is $k$-amenable if and only if either $c(T)_{k-1} = c(T)_k = 0$ or else it satisfies the following conditions:
\begin{enumerate}[(1)]
	\item $c(T)^{(u)}_{k-1} > c(T)^{(u)}_k$;
	\item when $T(x,y) = k$ then $|\mathcal{S}^{\boxtimes}_T(x,y)^{(k-1)}| \geq |\mathcal{S}^{\boxtimes}_T(x,y)^{(k)}|$;
	\item for each $(x,y) \in \mathcal{B}_T^{(k)}$ we have $|\mathcal{S}^{\boxtimes}_T(x,y)^{(k-1)}| > |\mathcal{S}^{\boxtimes}_T(x,y)^{(k)}|$;
	\item if $d = b_T^{(k)}+c^{(u)}_k-c^{(u)}_{k-1}+1 > 0$ then there is an injective map $\phi: \mathcal{B}_T^{(k)}(d) \rightarrow \widehat{\mathcal{B}_T^{(k-1)}}$ such that if $(x,y) \in \mathcal{B}_T^{(k)}(d)$ and $(u,v) = \phi(x,y)$ then for all $u < r < x$ we have $T(r,s) \notin \{k-1,k'\}$ for all $s$ such that $(r,s) \in D_{\lambda/\mu}$;
	\item $T^{(k-1)}$ is fitting;
  \item if $c(T)_k > 0$ then $T^{(k)}$ is fitting.
\end{enumerate}
\end{Lem}
\begin{Co}\cite[Corollary 2.15]{Schure}\label{checklistco}
Let $\lambda, \mu \in DP$.
Let $T$ be a tableau of shape $D_{\lambda/\mu}$ such that either $c(T)_k = c(T)_{k-1} = 0$ or else it satisfies the following conditions:
\begin{enumerate}[(1)]
	\item there is some box $(x,y)$ such that $T(x,y) = k-1$ and $T(z,y) \neq k$ for all $z > x$;
	\item if $T(x,y) = k$ then there is some $z < x$ such that $T(z,y) = k-1$;
	\item if $T(x,y) = k'$ then $T(x-1,y-1) = (k-1)'$;
	\item $T^{(k-1)}$ is fitting;
  \item if $c^{(u)}_k > 0$ then $T^{(k)}$ is fitting.
\end{enumerate}
Then the tableau is $k$-amenable.
\end{Co}

\begin{Ex}
We consider again the following tableau $T$ of shape $D_{(11,9,6,5,4,2,1)/(8,6,5,4,1)}$:
$${\Yvcentermath1 T = \young(\times \times \times \times \times \times \times \times \dmeins \deins \deins ,:\times \times \times \times \times \times \dmeins \dmzwei \dzwei ,::\times \times \times \times \times \deins ,:::\times \times \times \times \mzwei ,::::\times \meins 12,:::::1\mzwei ,::::::2)}\:.$$
We want to check the conditions of Lemma~\ref{checklist} for $k = 2$.
We have $c(T)^{(u)}_1 = 5 > 3 = c(T)^{(u)}_2$.
Since $T^{-1}(2) = \{(2,10), (5,8), (7,7)\}$ we need to check condition (2) of Lemma~\ref{checklist} for these boxes.
We have $|\mathcal{S}^{\boxtimes}_T(2,10)^{(1)}| = 2 \geq 1 = |\mathcal{S}^{\boxtimes}_T(2,10)^{(2)}|$, $|\mathcal{S}^{\boxtimes}_T(5,8)^{(1)}| = 3 \geq 2 = |\mathcal{S}^{\boxtimes}_T(5,8)^{(2)}|$ and $|\mathcal{S}^{\boxtimes}_T(7,7)^{(1)}| = 4 \geq 3 = |\mathcal{S}^{\boxtimes}_T(7,7)^{(2)}|$.
Since $\mathcal{B}_T^{(2)} = \{(2,9), (4,8)\}$ we need to check condition (3) of Lemma~\ref{checklist} for these boxes.
We have $|\mathcal{S}^{\boxtimes}_T(2,9)^{(1)}| = 2 > 1 = |\mathcal{S}^{\boxtimes}_T(2,9)^{(2)}|$ and $|\mathcal{S}^{\boxtimes}_T(4,8)^{(1)}| = 3 > 1 = |\mathcal{S}^{\boxtimes}_T(4,8)^{(2)}|$.
Since $d = 2+3-5+1 = 1$ we have to find a map as in condition (4) of Lemma~\ref{checklist} for the box $(2,9)$.
We have $\mathcal{B}_T^{(2)}(1) = \{(2,9)\}$ and $\widehat{\mathcal{B}_T^{(1)}} = \{(1,9), (2,8)\}$.
Both possible maps from $\mathcal{B}_T^{(2)}(1)$ to $\widehat{\mathcal{B}_T^{(1)}}$ satisfy the property of condition (4) of Lemma~\ref{checklist}.
Clearly, $T^{(1)}$ and $T^{(2)}$ are fitting.
Hence, the tableau $T$ is $2$-amenable.
\end{Ex}

In Section~4 we will start with a specific amenable tableau for a given diagram and change some entries to obtain new tableaux.
This specific tableau is obtained by an algorithm described by Salmasian in \cite[Section 3.1]{Salmasian}.
\begin{Def}\label{Salmasian'salgorithm}
Let $D_{\lambda/\mu}$ be a skew diagram.
The tableau $T_{\lambda/\mu}$ is determined by the following algorithm:
\begin{enumerate}[(1)]
	\item Set $k = 1$ and $U_1(\lambda/\mu) = D_{\lambda/\mu}$.
	\item Set $P_k = \{(x,y) \in U_k(\lambda/\mu) \mid (x-1,y-1) \notin U_k(\lambda/\mu)\}$.
	\item For each $(x,y) \in P_k$ set $T_{\lambda/\mu}(x,y) = k'$ if $(x+1,y) \in P_k$, otherwise set $T_{\lambda/\mu}(x,y) = k$.
	\item Let $U_{k+1}(\lambda/\mu) = U_k(\lambda/\mu)\setminus P_k$.
	\item Increase $k$ by one, and go to (2).
\end{enumerate}
\end{Def}

\begin{Ex}
For $\lambda = (6,5,3,2)$ and $\mu = (4,1)$ we obtain 
$$T_{\lambda/\mu} = {\Yvcentermath1 \young(::\meins 1,\meins 112,1\mzwei 2,:23)}.$$
\end{Ex}

The following definitions will be used in Proposition \ref{lambda/n}.

\begin{Def}
Let $\lambda \in DP$.
Then the \textbf{border} of $\lambda$ is defined by
$$B_{\lambda} := \{(x,y) \in D_{\lambda} \mid (x+1,y+1) \notin D_{\lambda}\}.$$
Define $B_{\lambda}^{(n)} := \{D_{\lambda/\mu} \mid D_{\lambda/\mu} \subseteq B_{\lambda} \text{ and } |D_{\lambda/\mu}| = n\}$.
\end{Def}

\begin{Def}\label{E}
Let $\lambda \in DP$.
Define $E_{\lambda}$ to be the set of all partitions whose diagram is obtained after removing a corner in $D_{\lambda}$.
\end{Def}

\begin{Prop}\cite[Proposition 3.2]{Schure}\label{lambda/n}
Let $\lambda \in DP$ and $1 \leq n \leq \lambda_1$ be an integer.
Then
$$Q_{\lambda/(n)} = \sum_{D_{\lambda/\nu} \in B_{\lambda}^{(n)} \hspace{1ex} (D_{\nu} \subseteq D_{\lambda})}{2^{comp(D_{\lambda/\nu})-1} Q_{\nu}}.$$
In particular, with $D_{\mu} = D_{\lambda} \setminus B_{\lambda}$ we have 
$$Q_{\lambda/(\lambda_1-1)} = \sum_{(x,y) \in B^{\times}_{\lambda}}{c^{(x,y)}_{B_{\lambda}} Q_{D_{\mu} \cup \{(x,y)\}}}$$
where $B^{\times}_{\lambda} := \{(x,y) \in B_{\lambda} \mid (x-1,y) \notin B_{\lambda} \text{ and } (x,y-1) \notin B_{\lambda}\}$ and
$$c^{(x,y)}_{B_{\lambda}} = \begin{cases}
1 &\mbox{if $(x,y)$ is the first or last box of $B_{\lambda}$} \\
2 &\mbox{otherwise},
\end{cases}$$
and
$$Q_{\lambda/(1)} = \sum_{\nu \in E_{\lambda}}{Q_{\nu}}.$$
\end{Prop}

\subsection{Equality of skew Schur $Q$-functions}

Before we analyze the $Q$-multiplicity-freeness of some given skew Schur $Q$-function $Q_{\lambda/\mu}$, we show that $Q_{\lambda/\mu} = Q_{\alpha/\beta}$ where $D_{\alpha/\beta}$ is a diagram obtained by
certain transformations of the diagram $D_{\lambda/\mu}$.
Hence it will be sufficient to analyze the skew Schur $Q$-function of one of these (transformed) diagrams to obtain statements for all of them.
This approach significantly reduces the effort in the proofs for the lemmas that lead to the classification of $Q$-multiplicity-free skew Schur $Q$-functions.

\begin{Def}
Let $D$ be an unshifted skew diagram.
The \textbf{transpose} of $D$, denoted by $D^t$, is the unshifted skew diagram obtained by reflecting the boxes of $D$ along the diagonal $\{(x,x) \mid x \in \NN\}$ and moving this arrangement of boxes such that the top row with boxes is in the first row and the lowest box of the leftmost column with boxes is on the diagonal $\{(x,x) \mid x \in \NN\}$.\\
The \textbf{rotation} of $D$, denoted by $D^o$, is the unshifted skew diagram obtained by rotating the boxes of $D$ through 180° and moving this arrangement of boxes such that the topmost row with boxes is in the first row and the lowest box of the leftmost column with boxes is on the diagonal $\{(x,x) \mid x \in \NN\}$.
\end{Def}

These are transformations of unshifted skew diagrams that do not change the corresponding $Q$-function, as is stated in the following lemma which is implied by \cite[Proposition 3.3]{BarekatvanWilligenburg} by Barekat and van Willigenburg.

\begin{Lem}\label{transposerotate}
Let $D=D_{\lambda/\mu}$ be an unshifted skew diagram.
There is a content-preserving bijection between tableaux of $D$ and tableaux of $D^t$,
as well as
a content-preserving bijection between the tableaux of $D$ and the tableaux of $D^o$.
Hence
$$Q_{\lambda/\mu} = Q_{(\lambda/\mu)^t} = Q_{(\lambda/\mu)^o}.$$
\end{Lem}

\begin{Def}
Let $D$ be a (not necessarily unshifted) skew diagram.
The \textbf{orthogonal transpose} of $D$, denoted  by $D^{ot}$, is obtained as follows: reflect the boxes of $D$ along the diagonal $\{(z,-z) \mid z \in \NN\}$.
Move this arrangement of boxes such that the top row with boxes is in the first row and the lowermost box of the leftmost column with boxes is on the diagonal $\{(z,z) \mid z \in \NN\}$.
\end{Def}
\begin{Ex}
For
$${\Yvcentermath1 D = \young(::\none \none \none ,\none \none \none \none \none ,\none \none \none \none ,:\none \none \none ,::\none )}$$
we obtain
$${\Yvcentermath1 D^{ot} = \young(:::\none \none ,:\none \none \none \none ,\none \none \none \none \none ,:\none \none \none ,::\none \none )}.$$
\end{Ex}

As it turns out this is again a transformation on shifted skew diagrams that leaves the respective $Q$-function unchanged; this has been shown by DeWitt (\cite[Proposition IV.13]{DeWitt}) and independently in \cite[Lemma 3.11]{Schure}.

\begin{Lem}\label{ot}
Let $D$ be a skew diagram. 
There is a content-preserving bijection between the tableaux of $D$ and the tableaux of $D^{ot}$.
\end{Lem}

The diagrams $U_i$ in the following lemma are defined by Salmasian's algorithm
as in Definition~\ref{Salmasian'salgorithm}.
\begin{Lem}\label{ot2}
Let $\lambda, \mu \in DP$, $\nu = c(T_{\lambda/\mu})$ and $n = \ell(\nu)$.
Let $D_{\lambda/\mu}^{ot}$ have shape $D_{\gamma/\delta}$.
Let $T' = T_{\gamma/\delta}$.
If $U_i(\lambda/\mu)$ has shape $D_{\alpha/\beta}$ then $U_i(\gamma/\delta)$ has shape $D_{\alpha/\beta}^{ot}$.
\end{Lem}
\begin{proof}
The diagram $U_i(\gamma/\delta)$ is also defined by $\{(x,y) \in D_{\gamma/\delta} \mid (x-i+1,y-i+1) \in D_{\gamma/\delta}\}$ and the image of this set of boxes after orthogonally transposing is given by the set of boxes $\{(u,v) \in D_{\lambda/\mu} \mid (u+i-1,v+i-1) \in D_{\lambda/\mu}\}$ which has the same shape as the set of boxes $\{(u,v) \in D_{\lambda/\mu} \mid (u-i+1,v-i+1) \in D_{\lambda/\mu}\} = U_i(\lambda/\mu)$.
\end{proof}
\begin{Rem}
For $i = n$ this means that $T'^{(n)}$ has the same shape as $P_n^{ot}$.
\end{Rem}

\section{Relations between the coefficients $f^{\lambda}_{\mu \nu}$}

In this section we will prove some inequalities satisfied by the shifted Littlewood-Richardson coefficients $f^{\lambda}_{\mu \nu}$ that will be helpful for the proofs in Section~4.

\begin{Lem}\label{parts}
Let $\lambda, \mu \in DP$.
Let $c(T_{\lambda/\mu})=\nu=(\nu_1,\ldots, \nu_n)$.
Let $k$ be such that $U_k(\lambda/\mu)$ has shape $D_{\alpha/\beta}$ for some $\alpha, \beta \in DP$.
Then
$$\prod_{j = 1}^{k-1}{2^{comp(P_j)-1}} f^{\alpha}_{\beta \gamma} = f^{\lambda}_{\mu (\nu_1, \ldots \nu_{k-1}, \gamma_1, \ldots, \gamma_{\ell(\gamma)})}.$$
In particular, we have $f^{\alpha}_{\beta \gamma} \leq f^{\lambda}_{\mu (\nu_1, \ldots \nu_{k-1}, \gamma_1, \ldots, \gamma_{\ell(\gamma)})}$.
\end{Lem}
\begin{proof}
Let $f^{\alpha}_{\beta \gamma} = m$, i.e.,
there are exactly $m$ different amenable tableaux of $D_{\alpha/\beta}$ of content $\gamma=(\gamma_1, \ldots, \gamma_{\ell(\gamma)})$.
Then we can obtain $\prod_{i = 1}^{k-1}{2^{comp(P_i)-1}} f^{\alpha}_{\beta \gamma}$ different amenable tableaux of $D_{\lambda/\mu}$ of content $(\nu_1, \ldots \nu_{k-1}, \gamma_1, \ldots, \gamma_{\ell(\gamma)})$ as follows.
For each box of $D_{\alpha/\beta}$ replace its entry $i$ (respectively, $i'$) by $i+k-1$ (respectively, $(i+k-1)'$).
Use these as the filling of the boxes of $U_k(\lambda/\mu)$.
For all $1 \leq j \leq k-1$ fill the boxes of $P_j$ with entries from $\{j', j\}$.
By Lemma \ref{2fillings}, we have $2^{comp(P_j)-1}$ possible fillings of $P_j$ such that $P_j$ is fitting.
We only need to show $k$-amenability for each of these tableaux, which follows straightforwardly by Corollary \ref{checklistco}.
Hence, $\prod_{j = 1}^{k-1}{2^{comp(P_j)-1}} f^{\alpha}_{\beta \gamma} \leq f^{\lambda}_{\mu (\nu_1, \ldots \nu_{k-1}, \gamma_1, \ldots, \gamma_{\ell(\gamma)})}$.\par
Let $T$ be an amenable tableaux of $D_{\lambda/\mu}$ of content $(\nu_1, \ldots \nu_{k-1}, \gamma_1, \ldots, \gamma_{\ell(\gamma)})$.
Then, by Lemma \ref{diagonal}, each box $(x,y)$ of the $\nu_1$ entries from $\{1', 1\}$ is such that $(x-1,y-1) \notin D_{\lambda/\mu}$.
The set of all such boxes is $P_1$ and we have $|P_1| = \nu_1$.
Hence, $T^{(1)} = P_1$.
Then, by Lemma \ref{diagonal}, each box $(x,y)$ of the $\nu_2$ entries from $\{2', 2\}$ satisfy the property $(x-1,y-1) \notin D_{\lambda/\mu} \setminus P_1$.
The set of all such boxes is $P_2$ and we have $|P_2| = \nu_2$.
Hence, $T^{(2)} = P_2$.
Repeating this argument for all entries greater than $2$, we see that $T^{(j)} = P_j$ for all $1 \leq j \leq k-1$.
Hence, after removing $T^{(1)}, T^{(2)}, \ldots, T^{(k-1)}$ we obtain some tableau of $U_k(\lambda/\mu)$ of shape $D_{\alpha/\beta}$.
If for each box of this tableau we replace its entry $i$ (respectively, $i'$) by $i-k+1$ (respectively, $(i-k+1)'$) then we obtain a tableau $T'$ of $D_{\alpha/\beta}$ of content $\gamma$.
The amenability of the tableau $T'$ follows from the amenability of the tableau $T$.
After removing the ribbon strips $P_j$ for all $1 \leq j \leq k-1$, the tableau $T'$ is independent of the different possible fillings of these $P_j$.
By Lemma \ref{2fillings}, we have $2^{comp(P_j)-1}$ possible fillings of $P_j$.
Thus, for each of the $\prod_{j = 1}^{k-1}{2^{comp(P_j)-1}}$ tableaux of content $(\nu_1, \ldots \nu_{k-1}, \gamma_1, \ldots, \gamma_{\ell(\gamma)})$ with the same entries in $U_k(\lambda/\mu)$ we obtain $T'$.
Hence, we have $\prod_{j = 1}^{k-1}{2^{comp(P_j)-1}} f^{\alpha}_{\beta \gamma} \geq f^{\lambda}_{\mu (\nu_1, \ldots \nu_{k-1}, \gamma_1, \ldots, \gamma_{\ell(\gamma)})}$ and the statement follows.
\end{proof}

We will use the relation $f^{\alpha}_{\beta \gamma} \leq f^{\lambda}_{\mu (\nu_1, \ldots \nu_{k-1}, \gamma_1, \ldots, \gamma_{\ell(\gamma)})}$ in
Section~4 to show that $f^{\lambda}_{\mu \delta} \geq 2$ for some $\delta$ by showing that $f^{\alpha}_{\beta \gamma} \geq 2$ for some $\gamma$.
Then by setting $\delta = (\nu_1, \ldots \nu_{k-1}, \gamma_1, \ldots, \gamma_{\ell(\gamma)})$ we obtain the desired assertion.\par
The following example illustrates how to obtain tableaux of shape $D_{\lambda/\mu}$ from the tableaux of shape $D_{\alpha/\beta}$ as explained in the proof of Lemma \ref{parts}.

\begin{Ex}
Let $\lambda = (10,8,7,6,4,1)$ and $\mu = (5,3,2,1)$ and consider
$${\Yvcentermath1 T_{\lambda/\mu} = \young(:\meins 11111,\meins 1\mzwei 222,\meins \mzwei 2\mdrei 33,\meins \mzwei \mdrei 3\mvier 4,1\mzwei 344,:2).}$$
Let $k = 3$.
Then $${\Yvcentermath1 U_3(\lambda/\mu) = \young(:\none \none \none ,\none \none \none \none ,\none \none \none )\:.}$$
Two amenable tableaux of the same content are
$${\Yvcentermath1 \young(:\meins 11,1122,233)}\:, \hspace{1ex} {\Yvcentermath1 \young(:111,\meins 222,133)}\:.$$
We obtain two amenable tableaux of $D_{\lambda/\mu}$ of the same content:
$${\Yvcentermath1 \young(:\meins 11111,\meins 1\mzwei 222,\meins \mzwei 2\mdrei 33,\meins \mzwei 3344,1\mzwei 455,:2)}, \hspace{1ex} {\Yvcentermath1 \young(:\meins 11111,\meins 1\mzwei 222,\meins \mzwei 2333,\meins \mzwei \mdrei 444,1\mzwei 355,:2)}.$$
\end{Ex}

\begin{Def}
Let $\lambda, \mu \in DP$, let $2 \leq a \leq \ell(\mu)+2$, and let $b \geq \ell(\lambda)$.
Let $\Gamma^{\rightarrow}_a(D_{\lambda/\mu})$ be the diagram obtained from $D_{\lambda/\mu}$ by shifting all boxes above the $a^{\textrm{th}}$ row one box to the right.
Let $\Gamma^{\downarrow}_b(D_{\lambda/\mu})$ be the diagram obtained from $D_{\lambda/\mu}$ by shifting all boxes $(x,y)$ such that $y < b$ one box down.
\end{Def}
\begin{Ex}
For $\lambda = (8,7,4,3,1)$ and $\mu = (5,2,1)$ we have
$${\Yvcentermath1 D_{\lambda/\mu} = \young(\times \times \times \times \times \none \none \none ,:\times \times \none \none \none \none \none ,::\times \none \none \none ,:::\none \none \none ,::::\none )}$$
and
$${\Yvcentermath1 \Gamma^{\rightarrow}_5(D_{\lambda/\mu}) = \young(\times \times \times \times \times \times \none \none \none ,:\times \times \times \none \none \none \none \none ,::\times \times \none \none \none ,:::\times \none \none \none ,::::\none )}, {\Yvcentermath1 \Gamma^{\downarrow}_6(D_{\lambda/\mu}) = \young(\times \times \times \times \times \times \none \none \none ,:\times \times \times \times \times \none \none \none ,::\times \times \none \none \none ,:::\times \none \none \none ,::::\none \none ,:::::\none )}.$$
\end{Ex}

\begin{Lem}\label{shiftrowcolumn}
Let $\lambda, \mu \in DP$ and let $2 \leq a \leq \ell(\mu)+2$ and $b \geq \ell(\lambda)$.
Let $\Gamma^{\rightarrow}_a(D_{\lambda/\mu})$ have shape $D_{\alpha/\beta}$ and let $\Gamma^{\downarrow}_b(D_{\lambda/\mu})$ have shape $D_{\tilde{\alpha}/\tilde{\beta}}$.\par
Then $f^{\lambda}_{\mu \nu} \leq f^{\alpha}_{\beta \nu}$ and $f^{\lambda}_{\mu \nu} \leq f^{\tilde{\alpha}}_{\tilde{\beta} \nu}$.
\end{Lem}
\begin{proof}
For every given amenable tableau $T$ of shape $D_{\lambda/\mu}$ one obtains an amenable tableau $\hat{T}$ of shape $D_{\alpha/\beta}$ by setting $\hat{T}(x,y) = T(x,y-1)$ for all $1 \leq x \leq a-1$ and $\hat{T}(x,y) = T(x,y)$ for all $x \geq a$ such that $(x,y) \in D_{\alpha/\beta}$.
Since $w(\hat{T}) = w(T)$, the tableau $\hat{T}$ is amenable.
Let $T, T'$ be two amenable tableaux of shape $D_{\lambda/\mu}$ and $\hat{T}, \hat{T'}$ the tableaux obtained from $T, T'$ as above.
If $\hat{T} = \hat{T'}$ then $T(x,y) = \hat{T}(x,y+1) = \hat{T'}(x,y+1) = T'(x,y)$ for all $1 \leq x \leq a-1$ and $T(x,y) = \hat{T}(x,y) = \hat{T'}(x,y) = T'(x,y)$ for all $x \geq a$ and, hence, $T = T'$.
Thus, the statement $f^{\lambda}_{\mu \nu} \leq f^{\alpha}_{\beta \nu}$ follows.\par
For every given amenable tableau $T$ of shape $D_{\lambda/\mu}$ one obtains an amenable tableau $\tilde{T}$ of shape $D_{\tilde{\alpha}/\tilde{\beta}}$ by setting $\tilde{T}(x,y) = T(x-1,y)$ for all $1 \leq y \leq b-1$ and $\tilde{T}(x,y) = T(x,y)$ for all $y \geq b$ such that $(x,y) \in D_{\tilde{\alpha}/\tilde{\beta}}$.
By Lemma \ref{checklist}, the tableau $\tilde{T}$ is amenable.
Let $T, T'$ be two amenable tableaux of shape $D_{\lambda/\mu}$ and $\tilde{T}, \tilde{T'}$ the tableaux obtained from $T, T'$ as above.
If $\tilde{T} = \tilde{T'}$ then $T(x,y) = \tilde{T}(x+1,y) = \tilde{T'}(x+1,y) = T'(x,y)$ for all $1 \leq y \leq b-1$ and $T(x,y) = \tilde{T}(x,y) = \tilde{T'}(x,y) = T'(x,y)$ for all $y \geq b$ and, hence, $T = T'$.
Thus, the statement $f^{\lambda}_{\mu \nu} \leq f^{\tilde{\alpha}}_{\tilde{\beta} \nu}$ follows.
\end{proof}
\begin{Rem}
The statement $f^{\lambda}_{\mu \nu} \leq f^{\alpha}_{\beta \nu}$ in the sense of Lemma \ref{shiftrowcolumn} appeared in the proof of \cite[Theorem 2.2]{Bessenrodt} and is, hence, due to Bessenrodt.
In the same proof the statement $f^{\lambda}_{\mu \nu} \leq f^{\tilde{\alpha}}_{\tilde{\beta} \nu}$ for $b = \mu_1+2$ can be found (without explicitly stating that $\mu_1+2 \geq \ell(\lambda)$ is required).
\end{Rem}

\begin{Lem}\label{add1}
Let $w$ be an amenable word.
Let $\tilde{w}$ be a word such that after removing one letter $1$ the word obtained is $w$ (this means that $\tilde{w}$ can be obtained from $w$ by adding a letter $1$).
Then $\tilde{w}$ is amenable.
\end{Lem}
\begin{proof}
The number of letters equal to $1$ in $\tilde{w}$ is greater than the number of letters equal to $1$ in $w$.
Then the word $\tilde{w}$ is not amenable only if there is some $j \geq n = \ell(\tilde{w})$ such that $m_1(j) = m_2(j)$ and $w_{j-n+1}$ is this added $1$.
But then for the word $w$ we have $m_1(j-2) < m_2(j-2)$; a contradiction to the amenability of the word $w$.
\end{proof}

\begin{Def}
Let $\alpha \in DP$ and $a \in \NN$ such that $a \leq \ell(\alpha)+1$.
Then
$$\alpha + (1^a) := (\alpha_1+1, \alpha_2+1, \ldots, \alpha_a+1,\alpha_{a+1}, \alpha_{a+2}, \ldots, \alpha_{\ell(\alpha)}).$$
\end{Def}

\begin{Lem}\label{addrow}
Let $\lambda, \mu \in DP$ and let $1 \leq a \leq \ell(\mu)+1$.
Then $f^{\lambda}_{\mu \nu} \leq f^{\lambda + (1^a)}_{\mu + (1^{a-1}), \nu + (1)}$.
\end{Lem}
\begin{proof}
For this proof we will assume that for a tableau of shape $D_{\lambda/\mu}$ the boxes of $D_{\mu}$ are not removed but instead are filled with $0$.
Given an amenable tableau $T$ of shape $D_{\lambda/\mu}$ we obtain an amenable tableau $\bar{T}$ of shape $D_{(\lambda + (1^a))/(\mu + (1^{a-1}))}$ as follows.
Insert a box with entry $0$ into each of the first $a-1$ rows such that the rows are weakly increasing from left to right and insert a box with entry $1$ into the $a^{\textrm{th}}$ row such that this row is weakly increasing from left to right.\par
The word $w(\bar{T})$ differs from $w(T)$ only by one added $1$.
By Lemma \ref{add1}, the word $w(\bar{T})$ is amenable.
Clearly, if $T \neq T'$ for some tableaux $T, T' \in T(\lambda/\mu)$ then $\bar{T} \neq \bar{T'}$.
\end{proof}
\begin{Rem}
Note that $\Gamma^{\rightarrow}_{a+1}(D_{\lambda/\mu}) \cup \{(a,a+\mu_a)\}$ has shape $D_{(\lambda + (1^a))/(\mu + (1^{a-1}))}$.\par
The proof of Lemma \ref{addrow} is inspired by the proof of  \cite[Theorem 3.1]{Gutschwager} where Gutschwager proved a similar statement for classical Littlewood-Richardson coefficients.
\end{Rem}

\begin{Lem}\label{addcolumn}
Let $\lambda, \mu \in DP$ and let $b \geq \ell(\lambda)$.
Let $(a,b-1)$ be the uppermost box of $D_{\lambda/\mu}$ in the $(b-1)^{\textrm{th}}$ column.
Let $\Gamma^{\downarrow}_b(D_{\lambda/\mu}) \cup \{(a,b-1)\}$ have shape $D_{\alpha/\beta}$.\par
Then $f^{\lambda}_{\mu \nu} \leq f^{\alpha}_{\beta, \nu + (1)}$.
\end{Lem}
\begin{proof}
Again, we assume that for a tableau of shape $D_{\lambda/\mu}$ the boxes of $D_{\mu}$ are not removed but instead are filled with $0$.
Given an amenable tableau $T$ of shape $D_{\lambda/\mu}$ we obtain an amenable tableau $\bar{T}$ of shape $D_{\alpha/\beta}$ as follows.
Insert a box with entry $0$ into each of the first $b-2$ columns such that the columns are weakly increasing from top to bottom.
If there is no $1'$ or $1$ in the $(b-1)^{\textrm{th}}$ column then insert a box with entry $1$ into the $(b-1)^{\textrm{th}}$ column such that this column is weakly increasing from top to bottom.
If there is a $1'$ or a $1$ in the $(b-1)^{\textrm{th}}$ column then insert a box with entry $1'$ into the $(b-1)^{\textrm{th}}$ column such that this column is weakly increasing from top to bottom.\par
Let $\hat{T}$ be the tableau defined by $\hat{T}(x,y) := T(x-1,y)$ for all $1 \leq y \leq b-1$ and $\hat{T}(x,y) = T(x,y)$ for all $y \geq b$ such that $(x,y) \in \Gamma^{\downarrow}_b(D_{\lambda/\mu})$.
By Lemma \ref{shiftrowcolumn}, the tableau $\hat{T}$ is amenable.
The word $w = w(\bar{T})$ differs from $w(\hat{T})$ only by an added $1'$ or an added $1$.
If a $1'$ is added then clearly, the tableau $\bar{T}$ is amenable.
If a $1$ is added then, by Lemma \ref{add1}, the word $w(\bar{T})$ is amenable.
Clearly, if $T \neq T'$ for some tableaux $T, T' \in T(\lambda/\mu)$ then $\bar{T} \neq \bar{T'}$.
\end{proof}

\section{$Q$-multiplicity-free skew Schur $Q$-functions}

With the tools provided in Section 3 we can finally start to prove results towards our desired classification.

\begin{Def}
A symmetric function $f \in span(Q_{\lambda} \mid \lambda \in DP)$ is called \textbf{$Q$-multiplicity-free} if the coefficients in the decomposition of $f$ into Schur $Q$-functions are from $\{0, 1\}$.
In particular, a skew Schur $Q$-function $Q_{\lambda/\mu}$ is called $Q$-multiplicity-free if $f^{\lambda}_{\mu \nu} \leq 1$ for all $\nu \in DP$.
\end{Def}

Lemma \ref{parts} will be crucial in this chapter.
It allows us to consider subdiagrams consisting of the boxes of $T_{\lambda/\mu}$ with entries bigger than some given $k$ and simplifies the proof of non-$Q$-multiplicity-freeness.
This follows from the fact that if the skew Schur $Q$-function to a diagram is not $Q$-multiplicity-free then the same must hold for each diagram that contains this diagram as a subdiagram of $T_{\lambda/\mu}$ of boxes with entries greater than $k$ for some $k$.\par
We will analyze diagrams $D_{\lambda/\mu}$ and show that the corresponding $Q_{\lambda/\mu}$  are not $Q$-multiplicity-free by finding two different amenable tableaux of the same content, derived by changing entries in the tableau $T_{\lambda/\mu}$ obtained in Definition \ref{Salmasian'salgorithm}.
Using this method, we are able to find many types of diagrams
$D_{\lambda/\mu}$ such that the corresponding $Q_{\lambda/\mu}$
are not $Q$-multiplicity-free; we will prove afterwards that all remaining diagrams do give $Q$-multiplicity-free skew Schur $Q$-functions,
hence then our classification is complete.

Let $\lambda, \mu \in DP$ and $\nu = c(T_{\lambda/\mu})$.
Lemma \ref{2fillings} and the definition of amenability, which requires $P_i$ to be fitting, state that $f^{\lambda}_{\mu \nu} = 1$ is only possible if each $P_i$ is connected.
\begin{Hyp}
From now on we will consider only diagrams that satisfy the property that each $P_i$ is connected.\par
\end{Hyp}

The following lemmas give restrictions for the border strips $P_i$; they will enable us to prove Lemma \ref{notthreecorners} and Corollary \ref{notthreecornersco} which lower the number of families of partitions $\lambda$ we have to consider to find $Q$-multiplicity-free skew Schur $Q$-functions.

\begin{Lem}\label{justhook}
Let $\lambda, \mu \in DP$.
Let $\nu = c(T_{\lambda/\mu})$ and $n = \ell(\nu)$.
If $P_n$ is neither a hook nor a rotated hook then $Q_{\lambda/\mu}$ is not $Q$-multiplicity-free.
\end{Lem}
\begin{proof}
By Lemma \ref{parts}, it is enough to find two amenable tableaux of $P_n$ of the same content.
Assume that the diagram $P_n$ is neither a hook nor a rotated hook.
Then we can find a subset of boxes of $P_n$ such that all but one of the boxes form a $(p,q)$-hook where $p, q \geq 2$ and there is either a single box above the rightmost box of the hook, or a single box to the left of the lowermost box of the hook.
By Lemmas \ref{transposerotate}, \ref{shiftrowcolumn}, \ref{addrow} and \ref{addcolumn}, it is enough to assume that $P_n$ has shape $D_{(4,2)/(2)}$.
Since $Q_{(4,2)/(2)} = Q_{(4)} + 2Q_{(3,1)}$, the statement follows.
\end{proof}

\begin{Lem}\label{threethreehook}
Let $\lambda, \mu \in DP$. Let $\nu = c(T_{\lambda/\mu})$ and $n = \ell(\nu) > 1$.
Let $P_n$ be a $(p,q)$-hook or a rotated $(p,q)$-hook where $p, q \geq 3$.
Suppose the last box of $P_{n-1}$ is not in the row directly above the row of the last box of $P_n$.
Then $Q_{\lambda/\mu}$ is not $Q$-multiplicity-free.
\end{Lem}
\begin{proof}
We may assume that $P_n$ is a $(p,q)$-hook where $p, q \geq 3$.
Otherwise, $P_n$ is a rotated $(p,q)$-hook where $p, q \geq 3$ and we may consider $D_{\lambda/\mu}^{ot}$ since if $D_{\lambda/\mu}^{ot}$ has shape $D_{\alpha/\beta}$ then, by Lemma \ref{ot2}, the set of boxes $T_{\alpha/\beta}^{(n)}$ is a $(q,p)$-hook where $p, q \geq 3$.\par
By Lemma \ref{parts}, we may assume that $n = 2$.
Let $(x,y)$ be the last box of $P_2$.
By Lemmas \ref{shiftrowcolumn}, \ref{addrow} and \ref{addcolumn}, we may assume that $(x,y-1)$ is the last box of $P_1$.
We get a new tableau $T$ if we set $T(x,y-1) = 3$, $T(x-1,y-1) = 1$, $T(x,y) = 3$, $T(x-1,y) = 2$ and $T(r,s) = T_{\lambda/\mu}(r,s)$ for every other box $(r,s) \in D_{\lambda/\mu}$.\par
By Corollary \ref{checklistco}, this tableau is $m$-amenable for $m \neq 3$.
We have $T(x,y-1) = 3$ but there is no $2$ in the $(y-1)^{\textrm{th}}$ column.
However, there are at least two $2$s with no $3$ below them in the first two boxes of $P_2$.
Hence, by Lemma \ref{checklist}, this tableau is amenable.\par
We get another tableau $T'$ if we set $T'(x,y) = 3$, $T'(x-1,y) = 3'$, $T'(x,y-1) = 2$, $T'(x-1,y-1) = 1$ and $T'(r,s) = T_{\lambda/\mu}(r,s)$ for every other box $(r,s) \in D_{\lambda/\mu}$.\par
By Corollary \ref{checklistco}, this tableau is $m$-amenable for $m \neq 2,3$.
Since there is a $1$ but no $2$ in the $y^{\textrm{th}}$ column, $2$-amenability follows.
We have $T'(x,y) = 3$ but there is no $2$ in the $y^{\textrm{th}}$ column.
Also, we have $T'(x-1,y) = 3'$ and $T'(x-2,y-1) \neq 2'$.
However, in the first two boxes of $P_n$ are $2$s with no $3$ below.
Additionally, there is another $2$ with no $3$ below in the $(y-1)^{\textrm{th}}$ column.
Thus, by Lemma \ref{checklist}, $3$-amenability follows.
\end{proof}
\begin{Ex}
For
$${\Yvcentermath1 T_{\lambda/\mu} = \young(\meins 111,\meins \mzwei 22,\meins \mzwei ,12)}$$
we obtain
$${\Yvcentermath1 T_1 = \young(\meins 111,\meins \mzwei 22,12,33)},\; \vspace{1ex} {\Yvcentermath1 T'_1 = \young(\meins 111,\meins \mzwei 22,1\mdrei ,23)}\:.$$
We have
$Q_{(7,6,3,2)/(3,2,1)} = Q_{(7,5)}+Q_{(7,4,1)}+Q_{(7,3,2)}+Q_{(6,5,1)}+2Q_{(6,4,2)}+Q_{(6,3,2,1)}+Q_{(5,4,3)}+Q_{(5,4,2,1)}$.
\end{Ex}

\begin{Lem}\label{Pnmid}
Let $\lambda, \mu \in DP$.
Let $\nu = c(T_{\lambda/\mu})$ and $n = \ell(\nu) \geq 2$.
Let there be some $k < n$ such that the last box of $P_k$ is in a row strictly lower than the last box of $P_n$ and some $i < n$ such that the first box of $P_i$ is in a column strictly to the right of the first box of $P_n$.
Then $Q_{\lambda/\mu}$ is not $Q$-multiplicity-free.
\end{Lem}
\begin{proof}
Let $k, i$ be maximal with respect to these conditions and let $j = \min\{k, i\}$.
By Lemma \ref{parts}, we may assume that $j = 1$.
First, we assume that $i \leq k$.
Then let $\bar{k}$ be minimal such that the last box of $P_{\bar{k}}$ is in a row strictly lower than the last box of $P_n$.
Let $(u,v)$ be the lowermost box in the rightmost column with a box of $P_{\bar{k}}$ in a row strictly lower than the last box of $P_{\bar{k}+1}$.
Let $x=u-\bar{k}+i$ and $y=v-\bar{k}+i$.
Then $(x,y)$ is the lowermost box of $P_i$ in the $y^{\textrm{th}}$ column.
We get a new tableau $T$ if after the $(i-1)^{\textrm{th}}$ step of the algorithm of Definition \ref{Salmasian'salgorithm} we use $P'_i := P_i \setminus\{(x,y)\}$ instead of $P_i$.\par
Let $P'_z = T^{(z)}$.
Then for $i+1 \leq r \leq \bar{k}$ if $(x+r-i,y+r-i) \in P_r$ then we have $(x+r-i-1,y+r-i-1) \in P'_r$.
Hence, $(x,y) \in P'_{i+1}$.
Clearly, by Corollary \ref{checklistco}, this tableau is $m$-amenable for $m \neq i+1$.
We possibly have $T(x,y) = (i+1)'$ and $T(x-1,y-1) \neq i'$.
But there is an $i$ with no $i+1$ below in the column of the first box of $P_i$.
Thus, by Lemma \ref{checklist}, $(i+1)$-amenability follows.\par
Let $(c,d)$ be the last box of $P_{\bar{k}+1}$.
We get another tableau $T'$ with the same content if we set $T'(c,d) = (\bar{k}+1)'$ and $T'(e,f) = T_{\lambda/\mu}(e,f)$ for every other box $(e,f) \in D_{\lambda/\mu}$\par
By Corollary \ref{checklistco}, it is clear that $T'$ is amenable if $T$ is and we have $c(T') = c(T) = (\nu_1, \ldots, \nu_{i-1}, \nu_i-1, \nu_{i+1}, \ldots, \nu_{\bar{k}},\nu_{\bar{k}+1}+1, \nu_{\bar{k}+2}, \ldots \nu_n)$.\par
If $k \leq i$ then $U_k(\lambda/\mu)$ is an unshifted diagram and we showed that two amenable tableaux of $U_k(\lambda/\mu)^t$ with the same content exist.
By Lemma \ref{transposerotate}, the statement follows.
\end{proof}
\begin{Ex}
For
$${\Yvcentermath1 T_{\lambda/\mu} = \young(\meins 1111,\meins \mzwei 22,1\mzwei 33,:2)}$$
we obtain $${\Yvcentermath1 T = \young(\meins 1111,1\mzwei 22,2233,:3)}, \hspace{1ex} {\Yvcentermath1 T' = \young(\meins 1111,1\mzwei 22,22\mdrei 3,:3)}.$$
We have $Q_{(7,5,4,1)/(2,1)} = Q_{(7,5,2)}+Q_{(7,4,3)}+Q_{(7,4,2,1)}+2Q_{(6,5,3)}+Q_{(6,5,2,1)}+Q_{(6,4,3,1)}$.
\end{Ex}

\begin{Lem}\label{cornerfirst}
Let $\lambda, \mu \in DP$.
Let $\nu = c(T_{\lambda/\mu})$ and $n = \ell(\nu) > 1$.
Let there be some $k < n$ such that there is a corner, $(x,y)$ say, in $P_k$ above the boxes of $P_n$ and let there be some $i \leq k$ such that the first box of $P_i$ is above the $(x-k+i)^{\textrm{th}}$ row.
Then $Q_{\lambda/\mu}$ is not $Q$-multiplicity-free.
\end{Lem}
\begin{proof}
Let $k$ be minimal and $i$ be maximal with respect to these conditions.
Then for all $i+1 \leq a \leq k$ the first box of $P_a$ has no box of $P_a$ below.
Let $(x-k+a,y)$ be the first box of $P_a$ for $i+1 \leq a \leq k-1$ and let $(x-k+i,y)$ be the rightmost box of $P_i$ in the $(x-k+i)^{\textrm{th}}$ row.
We get a new tableau $T$ if we set $T(x-k+i,y) = i+1$, $T(x-k+i-1,y) = i$, for all $i+1 \leq a \leq k$ set $T(x-k+a,y) = a+1$, $T(x,y) = k+1$ and $T(u,v) = T_{\lambda/\mu}(u,v)$ for every other box $(u,v) \in D_{\lambda/\mu}$.
By Corollary \ref{checklistco}, this tableau is amenable.\par
We get a new tableau $T'$ if we set $T'(x,y) = (k+1)'$ and $T'(u,v) = T(u,v)$ for every other box $(u,v) \in D_{\lambda/\mu}$.
We have $T'(x,y) = (k+1)'$ and $T'(x-1,y-1) \neq k'$.
However, we have $T'(x-1,y) = k$ and there is no $k+1$ in the $y^{\textrm{th}}$ column.
Hence, by Lemma \ref{checklist}, $T'$ is $m$-amenable for all $m$.\par
Clearly, we have that $c(T) = c(T') = (\nu_1, \ldots, \nu_{i-1}, \nu_i-1, \nu_{i+1} \ldots, \nu_k, \nu_{k+1}+1,\linebreak \nu_{k+2}, \ldots, \nu_n)$.
\end{proof}
\begin{Ex}
For
$${\Yvcentermath1 T_{\lambda/\mu} = \young(:::\meins ,1111,:22)}$$
we obtain
$${\Yvcentermath1 T = \young(:::1,1112,:22)}\:, \hspace{1ex} {\Yvcentermath1 T' = \young(:::1,111\mzwei ,:22)}\:.$$
We have $Q_{(5,4,2)/(4)} = Q_{(5,2)}+2Q_{(4,3)}+Q_{(4,2,1)}$.\par
For
$${\Yvcentermath1 T_{\lambda/\mu} = \young(:::::\meins ,111111,:22222,::3333,:::44,::::5)}$$
we obtain
$${\Yvcentermath1 T = \young(:::::1,111112,:22223,::3334,:::44,::::5)}\:, \hspace{1ex} {\Yvcentermath1 T' = \young(:::::1,111112,:22223,::333\mvier ,:::44,::::5)}\:.$$
We have $Q_{(7,6,5,4,2,1/(6)} = Q_{(7,5,4,2,1)}+2Q_{(6,5,4,3,1)}$.
\end{Ex}

\begin{Lem}\label{almosthook2}
Let $\lambda, \mu \in DP$.
Let $\nu = c(T_{\lambda/\mu})$ and $n = \ell(\nu) > 1$.
Let there be some $k > 1$ such that the first box of $P_{k-1}$ is to the right of the column of first box of $P_k$, and $P_{k-1}$ is not a hook.
Then $Q_{\lambda/\mu}$ is not $Q$-multiplicity-free.
\end{Lem}
\begin{proof}
Let $k$ be maximal with respect to this property.
By Lemma \ref{parts}, we may assume that $k = 2$.
If the first box of $P_1$ is not a corner then, by Lemma \ref{cornerfirst}, $Q_{\lambda/\mu}$ is not $Q$-multiplicity-free.
Thus, consider that the first box of $P_1$ is a corner.
If the first box of $P_1$ is not in the row above the first box of $P_2$ then an orthogonally transposed version of Lemma \ref{cornerfirst} states that $Q_{\lambda/\mu}$ is not $Q$-multiplicity-free.
Since $P_1$ is not a hook, there are $v, w$ such that the boxes $(v-1,w), (v,w), (v,w-1) \in P_1$ and the first box of $P_1$ is not in the $w^{\textrm{th}}$ column.
Let $v$ be maximal with respect to this property.\par
We get a new tableau $T$ if we use $P'_1 := P_1 \setminus \{(v,w)\}$ instead of $P_1$ in the algorithm of Definition \ref{Salmasian'salgorithm}.
By Corollary \ref{checklistco}, it is clear that $T$ is $i$-amenable for $i \neq 2$.
We possibly have $T(v,w) = 2'$ and $T(v-1,w-1) \neq 1'$.
However, in the column containing the first box of $P_1$ there is a $1$ and no $2$.
Thus, by Lemma \ref{checklist}, this tableau is amenable.\par
We get another tableau $T'$ if we set $T'(v-1,w) = 1'$ and $T'(r,s) = T(r,s)$ for every other box $(r,s) \in D_{\lambda/\mu}$.
By Corollary \ref{checklistco}, $T'$ is $i$-amenable for $i \neq 2$.
There is a $2$ but no $1$ in the $w^{\textrm{th}}$ column.
However, in the column containing the first box of $P_1$ there is a $1$ and no $2$.
We possibly have $T(v,w) = 2'$ and $T(v-1,w-1) \neq 1'$.
However, we have $T(v-1,w) = 1'$.
Thus, by Lemma \ref{checklist}, $T'$ is amenable.\par
It is easy to see that $c(T) = c(T')$.
\end{proof}
\begin{Ex}
For $${\Yvcentermath1 T_{\lambda/\mu} = \young(:\meins 1111,:\meins \mzwei 22,11233)}$$
and $k = 2$ we obtain
$${\Yvcentermath1 T = \young(:\meins 1111,:1\mzwei 22,12233)}, \hspace{1ex} {\Yvcentermath1 T' = \young(:\meins 1111,:\meins \mzwei 22,12233)}.$$
We have $Q_{(8,6,5)/(3,2)} = Q_{(8,4,2)}+2Q_{(7,5,2)}+2Q_{(7,4,3)}+2Q_{(6,5,3)}$.\par
For
$${\Yvcentermath1 T_{\lambda/\mu} = \young(:\meins 1111,:\meins \mzwei 22,11\mzwei \mdrei 3,:2234)}$$
and $k = 2$ we obtain
$${\Yvcentermath1 T = \young(:\meins 1111,:1\mzwei 22,1\mzwei 2\mdrei 3,:2334)}, \hspace{1ex} {\Yvcentermath1 T' = \young(:\meins 1111,:\meins \mzwei 22,1\mzwei 2\mdrei 3,:2334)}.$$
We have $Q_{(8,6,5,4)/(3,2)} = Q_{(8,6,3,1)}+Q_{(8,5,4,1)}+Q_{(8,5,3,2)}+2Q_{(7,6,4,1)}+2Q_{(7,6,3,2)}+2Q_{(7,5,4,2)}$.
\end{Ex}

\begin{Lem}\label{fatcolumn}
Let $\lambda, \mu \in DP$.
Let $\nu = c(T_{\lambda/\mu})$ and $n = \ell(\nu) > 1$.
Let $P_n$ be a $(p,q)$-hook where $p,q \geq 2$ and let $(x,y)$ be the first box of $P_n$.
Let there be some $k < n$ and some $i \geq y$ such that there are at least two boxes of $P_k$ in the $i^{\textrm{th}}$ column.
Then $Q_{\lambda/\mu}$ is not $Q$-multiplicity-free.
\end{Lem}
\begin{proof}
Let $k$ be maximal with respect to this property.
Let $(u,v)$ be the lowermost box of $P_k$ in the $i^{\textrm{th}}$ column and let $(a_r,b_r)$ be the first box of $P_r$ for all $r$.
We get a new tableau $T$ if we set $T(u,v) = k+1$, $T(u-1,v) = k$, for all $k+1 \leq r \leq n$ set $T(a_r,b_r) = r+1$ and $T(c,d) = T_{\lambda/\mu}(c,d)$ for every other box $(c,d) \in D_{\lambda/\mu}$.
By Corollary \ref{checklistco}, $T$ is amenable.\par
Let $(e,f)$ be the last box of $P_n$ and let $(x-1,z)$ be the rightmost box of $P_{n-1}$ in the $(x-1)^{\textrm{th}}$ row.
We get another tableau $T'$ if we set $T'(e,f) = n+1$,\linebreak
$T'(e-1,f) = n$, $T'(a_n,b_n) = n$, $T'(x-1,z) = n'$ and $T'(c,d) = T(c,d)$ for every other box $(c,d) \in D_{\lambda/\mu}$.
By Corollary \ref{checklistco}, $T'$ is $m$-amenable for $m \neq n$.
We have $T'(x-1,z) = n'$ and $T'(x-2,z-1) \neq (n-1)'$.
However, if $(g,h)$ is the last box of $P_n$ then we have $T'(g-2,h-1) = (n-1)'$ and $T'(g-1,h) \neq n'$.
Thus, by Lemma \ref{checklist}, amenability follows.
\end{proof}
\begin{Ex}
For
$${\Yvcentermath1 T_{\lambda/\mu} = \young(:::\meins ,\meins 111,1\mzwei 22,:2\mdrei 3,::3)}$$
we obtain
$${\Yvcentermath1 T = \young(:::1,\meins 112,1\mzwei 23,:2\mdrei 4,::3)}\:, \hspace{1ex} {\Yvcentermath1 T' = \young(:::1,\meins 112,1\mzwei 2\mdrei ,:233,::4)}\:.$$
We have $Q_{(6,5,4,3,1)/(5,1)} = Q_{(6,4,3)}+Q_{(6,4,2,1)}+2Q_{(5,4,3,1)}$.
\end{Ex}
\begin{Co}\label{fatcolumnco}
Let $\lambda, \mu \in DP$.
Let $\nu = c(T_{\lambda/\mu})$ and $n = \ell(\nu) > 1$.
Let $P_n$ be a $(p,q)$-hook where $p,q \geq 2$ and let $(x,y)$ be the first box of $P_n$.
Let there be some $k < n$ and some $i \geq x$ such that there are at least two boxes of $P_k$ in the $i^{\textrm{th}}$ row.
Then $Q_{\lambda/\mu}$ is not $Q$-multiplicity-free.
\end{Co}

Now we are able to show an intermediate result that bounds the number of corners of $D_{\lambda/\mu}$ in the case of $Q$-multiplicity-freeness and, hence, of $D_{\lambda}$ for $\mu \neq \emptyset, (1)$.
The number of corners of $D_{\mu}$ is then also bounded for most $D_{\lambda/\mu}$ because of orthogonal transposition.
This restricts the number of cases we have to analyze.

\begin{Lem}\label{notthreecorners}
Let $\lambda, \mu \in DP$ where $\mu \neq \emptyset, (1)$.
If $D_{\lambda}$ has more than two corners then $Q_{\lambda/\mu}$ is not $Q$-multiplicity-free.
\end{Lem}
\begin{proof}
Assume $D_{\lambda}$ has more than two corners, $\mu \neq \emptyset, (1)$, and $Q_{\lambda/\mu}$ is $Q$-multiplicity-free.
We will construct two amenable tableaux of shape $D_{\lambda/\mu}$ and of the same content to arrive at a contradiction.
Let $\nu = c(T_{\lambda/\mu})$ and $n = \ell(\nu)$.
Let $k$ be maximal such that $U_k(\lambda/\mu)$ has at least three corners.
Thus, at least one corner is in $P_k$.
By Lemma \ref{justhook} $P_n$ must be a hook or a rotated hook, so $P_n$ can have at most two corners and, hence, $k < n$.
By Lemma \ref{Pnmid} either the uppermost or the lowermost corner must be in $P_n$, so we only consider diagrams such that the uppermost or the lowermost corner is in $P_n$.
Without loss of generality we may assume that the lowermost corner of $U_k(\lambda/\mu)$ is in $P_n$, otherwise $U_k(\lambda/\mu)$ is an unshifted diagram and we may transpose $U_k(\lambda/\mu)$.
Thus, the uppermost corner is in $P_k$.
By Lemma \ref{cornerfirst}, which forbids to have boxes of $P_k$ to the left and above a corner in $P_k$ at the same time, the uppermost corner is the first box of $P_k$ and it is the only corner of the diagram $U_k(\lambda/\mu)$ that is in $P_k$.\par
Case 1: two corners are in $P_n$.\par
Then $P_n$ is a $(p,q)$-hook where $p \geq 2$ and $q \geq 2$.
By Lemma \ref{fatcolumn} and Corollary \ref{fatcolumnco}, which in this case for all $k \leq i \leq n-1$ forbid to have more than one box of $P_i$ in the column of the first box of $P_n$ and in the row of the last box of $P_n$, all $P_i$ are hooks.\par
Case 1.1: the last box of $P_{n-1}$ is in the same row as the last box of $P_n$.\par
Let $(u_a,v_a)$ be the last box of $P_a$ for all $a$.
We get a new tableau $T_1$ if for all $k \leq a \leq n$ we set $T_1(u_a,v_a) = a+1$, $T_1(u_a-1,v_a) = a$ and $T_1(r,s) = T_{\lambda/\mu}(r,s)$ for every other box $(r,s) \in D_{\lambda/\mu}$.
By Corollary \ref{checklistco}, $T_1$ is $m$-amenable for $m \neq k+1$.
Also by Corollary \ref{checklistco}, the tableau $T_1$ is also $(k+1)$-amenable because in the column of the first box of $P_k$ is a $k$ and no $k+1$.\par
We get another tableau $T'_1$ if we set $T'_1(u_n-1,v_n) = n'$ and $T'_1(r,s) = T_1(r,s)$ for every other box $(r,s) \in D_{\lambda/\mu}$.
By Corollary \ref{checklistco}, $T'_1$ is $m$-amenable for $m \neq n+1$.
We have $T'_1(u_n,v_n) = n+1$ and $T'_1(u_n-1,v_n) < n$, however, there is an $n$ with no $n+1$ below in the first box of $P_n$, and we have $T'_1(u_{n-1},v_{n-1}) = n$.
Thus, by Lemma \ref{checklist}, $(n+1)$-amenability follows.
We have $c(T_1) = c(T'_1)$.\par
Case 1.2: the last box of $P_{n-1}$ is in the row above the row of the last box of $P_n$.\par
For $p = 2$ we get $\mu = (1)$, which is a contradiction.
Thus, we have $p > 2$.
Let $(u_a,v_a)$ be the last box of $P_a$ for all $a$.
We get a new tableau $T_2$ if we set $T_2(u_n,v_n) = n+1$, $T_2(u_n-1,v_n) = (n+1)'$, for all $k \leq a \leq n-1$ set $T_2(u_a,v_a) = a+1$, $T_1(u_a-1,v_a) = a$ and $T_2(r,s) = T_{\lambda/\mu}(r,s)$ for every other box $(r,s) \in D_{\lambda/\mu}$.
By Corollary \ref{checklistco}, $T_2$ is $m$-amenable for $m \neq n+1$.
We have $T_2(u_n-1,v_n) = (n+1)'$ and $T_2(u_n-2,v_n-1) \neq n'$.
However, we have $T_2(u_n-2,v_n) = n'$.
Thus, by Lemma \ref{checklist}, $(n+1)$-amenability follows.\par
We get another tableau $T'_2$ if we set $T'_2(u_n-2,v_n) = n$ and $T'_2(r,s) = T_2(r,s)$ for every other box $(r,s) \in D_{\lambda/\mu}$.
By Lemma \ref{checklist}, it is clear that $T'_2$ is amenable if $T_2$ is amenable.
We have $c(T_2) = c(T'_2)$.\par
Case 2: only one corner is in $P_n$.\par
Let the second uppermost corner be in $P_i$.
Then by Lemma \ref{cornerfirst}, the second uppermost corner is the first box of $P_i$ and the uppermost corner is the first box of $P_k$.
If $P_i$ has all boxes in a row then $\mu = \emptyset$; a contradiction.
Thus, the diagram $P_i$ has at least two corners.
By Lemma \ref{almosthook2}, $P_i$ is a hook.
Then for all $i \leq j < n$ each $P_j$ is a $(p,q)$-hook for some $p,q \geq 2$.\par
Case 2.1: The last box of $P_{i-1}$ is in the same row as the last box of $P_i$.\par
Let $(g,h)$ be the last box of $P_i$ and $(c_a,d_a)$ be the rightmost box of $P_a$ in the lowermost row with boxes from $P_a$ for all $k \leq a \leq i-1$.
We get a new tableau $T_3$ if for all $k \leq a \leq i-1$ we set $T_3(c_a,d_a) = a+1$ if $(c_a+1,d_a) \notin D_{\lambda/\mu}$ or else set $T_3(c_a,d_a) = (a+1)'$ if $(c_a+1,d_a) \in D_{\lambda/\mu}$,
set $T_3(c_a-1,d_a) = a$, $T_3(g,h) = i+1$ and $T_3(r,s) = T_{\lambda/\mu}(r,s)$ for every other box $(r,s) \in D_{\lambda/\mu}$.
By Corollary \ref{checklistco}, the tableau $T_3$ is $m$-amenable for $m \neq k+1, i+1$.
We possibly have $T_3(c_k,d_k) = (k+1)'$ and $T_3(c_k-1,d_k-1) \neq k'$.
If not, then there is possibly a $k+1$ in the ${d_k}^{\textrm{th}}$ column.
Anyway, there is a $k$ with no $k+1$ below in the first box of $P_k$.
Thus, by Lemma \ref{checklist}, $(k+1)$-amenability follows.
We have $T_3(g,h) = i+1$ and $T_3(g-1,h) < i$.
However, there is an $i$ with no $i+1$ below in the first box of $P_i$.
Thus, by Lemma \ref{checklist}, $(i+1)$-amenability follows.\par
We get another tableau $T'_3$ if we set $T'_3(g-1,h) = i$ and $T'_3(r,s) = T_3(r,s)$ for every other box $(r,s) \in D_{\lambda/\mu}$.
Clearly, $T'_3$ is amenable if $T_3$ is and we have $c(T_3) = c(T'_3)$.\par
Case 2.2: The last box of $P_{i-1}$ is in the row above the row of the last box of $P_i$.\par
If in the column of the last box of $P_i$ there are only two boxes of $P_i$ then we have $\mu = (1)$, which is a contradiction.
Thus, there are at least three boxes of $P_i$ in the column of the last box of $P_i$.
Let $(c_a,d_a)$ be the last box of $P_a$ for all $k \leq a \leq i+1$.
We get a new tableau $T_4$ if for all $k \leq a \leq i-1$ we set $T_4(c_a,d_a) = a+1$, $T_4(c_a-1,d_a) = a$, $T_4(c_i,d_i) = i+1$, $T_4(c_i-1,d_i) = (i+1)'$, $T_4(c_{i+1},d_{i+1}) = i+2$, $T_4(c_{i+1}-1,d_{i+1}) = i+1$ and $T_4(r,s) = T_{\lambda/\mu}(r,s)$ for every other box $(r,s) \in D_{\lambda/\mu}$.\par
By Corollary \ref{checklistco}, the tableau $T_4$ is $m$-amenable for $m \neq k+1, i+1$.
There is a $k$ with no $k+1$ below in the first box of $P_k$.
Thus, by Corollary \ref{checklistco}, $(k+1)$-amenability follows.
We have $T_4(c_i,d_i) = i+1$ and there is no $i$ in the ${d_i}^{\textrm{th}}$ column.
However, there is an $i$ with no $i+1$ below in the first box of $P_i$.
We have $T_4(c_i-1,d_i) = (i+1)'$ and $T_4(c_i-2,d_i-1) \neq i'$.
However, we have $T_4(c_i-2,d_i) = i'$.
Thus, by Lemma \ref{checklist}, $(i+1)$-amenability follows.\par
We get another tableau $T'_4$ if we set $T'_4(c_i-2,d_i) = i$ and $T'_4(r,s) = T_4(r,s)$ for every other box $(r,s) \in D_{\lambda/\mu}$.\par
The tableau $T'_4$ is $m$-amenable for $m \neq i+1$.
We have $T'_4(c_i-1,d_i) = (i+1)'$ and $T'_4(c_i-2,d_i-1) \neq i'$.
However, there is an $i$ with no $i+1$ below in the first box of $P_i$.
Thus, by Lemma \ref{checklist}, $(i+1)$-amenability follows.
We have $c(T_4) = c(T'_4)$.
\end{proof}
\begin{Ex}
For
$${\Yvcentermath1 T_{\lambda/\mu} = \young(\meins 111111,\meins \mzwei 2222,\meins \mzwei \mdrei 333,1\mzwei \mdrei \mvier 44,:2\mdrei \mvier \mfuenf5,::345)}$$
we obtain
$${\Yvcentermath1 T_1 = \young(\meins 111111,\meins \mzwei 2222,1\mzwei \mdrei 333,22\mdrei \mvier 44,:33455,::456)}, \hspace{1ex} {\Yvcentermath1 T'_1 = \young(\meins 111111,\meins \mzwei 2222,1\mzwei \mdrei 333,22\mdrei \mvier 44,:334\mfuenf 5,::456)}.$$
We have $Q_{(10,8,7,6,5,3)/(3,2,1)} = Q_{(10,8,7,5,3)}+Q_{(10,8,7,5,2,1)}+Q_{(10,8,7,4,3,1)}+Q_{(10,8,6,5,3,1)}+Q_{(9,8,7,6,3)}+Q_{(9,8,7,6,2,1)}+Q_{(9,8,7,5,4)}+3Q_{(9,8,7,5,3,1)}+Q_{(9,8,7,4,3,2)}+Q_{(9,8,6,5,4,1)}+Q_{(9,8,6,5,3,2)}$.\par
For $${\Yvcentermath1 T_{\lambda/\mu} = \young(\meins 1111,\meins \mzwei 22,1\mzwei \mdrei 3,:2\mdrei ,::3)}$$
we obtain
$${\Yvcentermath1 T_2 = \young(\meins 1111,1\mzwei 22,22\mdrei 3,:3\mvier ,::4)}, \hspace{1ex} {\Yvcentermath1 T'_2 = \young(\meins 1111,1\mzwei 22,2233,:3\mvier ,::4)}.$$
We have $Q_{(7,5,4,2,1)/(2,1)} = Q_{(7,5,4)}+Q_{(7,5,3,1)}+Q_{(6,4,3,2,1)}+2Q_{(6,5,3,2)}+Q_{(7,4,3,2)}+Q_{(6,5,4,1)}$.\par
For
$${\Yvcentermath1 T_{\lambda/\mu} = \young(\meins 1111,\meins \mzwei 22,\meins \mzwei \mdrei ,123)}$$
we obtain
$${\Yvcentermath1 T_3 = \young(\meins 1111,\meins \mzwei 22,1\mzwei \mdrei ,233)}, \hspace{1ex} {\Yvcentermath1 T'_3 = \young(\meins 1111,\meins \mzwei 22,12\mdrei ,233)}.$$
We have $Q_{(8,6,4,3)/(3,2,1)} = Q_{(8,5,2)}+Q_{(8,4,3)}+Q_{(7,6,2)}+Q_{(8,4,2,1)}+2Q_{(7,5,3)}+Q_{(6,4,3,2)}+2Q_{(6,5,3,1)}+Q_{(6,5,4)}+2Q_{(7,4,3,1)}+2Q_{(7,5,2,1)}$.\par
For
$${\Yvcentermath1 T_{\lambda/\mu} = \young(\meins 1111,\meins \mzwei 22,1\mzwei \mdrei ,:23)}$$
we obtain
$${\Yvcentermath1 T_4 = \young(\meins 1111,1\mzwei 22,2\mdrei 3,:34)}, \hspace{1ex} {\Yvcentermath1 T'_4 = \young(\meins 1111,1222,2\mdrei 3,:34)}.$$
We have $Q_{(7,5,3,2)/(2,1)} = Q_{(7,5,2)}+Q_{(7,4,3)}+Q_{(7,4,2,1)}+Q_{(6,5,3)}+Q_{(6,5,2,1)}+2Q_{(6,4,3,1)}+Q_{(5,4,3,2)}$.
\end{Ex}
\begin{Co}\label{notthreecornersco}
Let $\lambda, \mu \in DP$.
Let $\nu = c(T_{\lambda/\mu})$ and $n = \ell(\nu) > 1$.
If $D_{\lambda/\mu}^{ot}$ has shape $D_{\alpha/\beta}$ where $\beta \neq \emptyset, (1)$ and $D_{\alpha}$ has more than two corners then $Q_{\lambda/\mu}$ is not $Q$-multiplicity-free.
If $D_{\lambda/\mu}$ is an unshifted diagram and $D^o_{\lambda/\mu}$ has more than two corners then $Q_{\lambda/\mu}$ is not $Q$-multiplicity-free.
\end{Co}

\begin{Rem}
As it will turn out (and will be proved in Lemma \ref{casea}), for $\mu = \emptyset$ or $\mu = (1)$ the skew Schur $Q$-function $Q_{\lambda/\mu}$ is $Q$-multiplicity-free.
Thus, we will only consider the case $\mu \neq \emptyset, (1)$.
Since we want to find all $\lambda, \mu$ such that $Q_{\lambda/\mu}$ is $Q$-multiplicity-free, by Lemma \ref{notthreecorners} from now on we will assume that $\lambda$ has at most two corners.
\end{Rem}

The case that the diagram $D_{\lambda}$ or the diagram $D_{\mu}$ has at most two corners also occurs in the classical setting of Schur functions $s_{\lambda/\mu}$.
Gutschwager proved \cite[Theorem 3.5]{Gutschwager} where the cases in condition (2) have this property.
However, this property is not enough in the classical case, where further restrictions need to be imposed for the classification of (s-)multiplicity-free skew Schur functions.
For the classification of $Q$-multiplicity-free skew Schur $Q$-functions we also need to find further restrictions since the properties from Lemma \ref{notthreecorners} and Corollary \ref{notthreecornersco} are not sufficient.\par
We will introduce some new notation for partitions in $DP$ with at most two corners and then obtain restrictions until we can exclude all non-$Q$-multiplicity-free skew Schur $Q$-functions in Proposition \ref{mf1}.

\begin{Def}\label{shapepath}
Let $DP^{\leq 2} \subseteq DP$ be the set of partitions $\lambda$ with distinct parts such that $D_{\lambda}$ has at most two corners.
For $\lambda \in DP^{\leq 2}$ the \textbf{shape path} $\pi_\lambda$ is defined as follows.
Let $a$ be the row of the upper corner of $D_{\lambda}$.
Let
$$b = \begin{cases}
\lambda_a \text{ if } a = \ell(\lambda);\\
\lambda_a - \lambda_{a+1} -1 \text{ otherwise}.
\end{cases}$$
If there is a lower corner in $D_{\lambda}$ let $c = \ell(\lambda)-a$ and $d = \lambda_{\ell(\lambda)}$.
Then the shape path to~$\lambda$ is $\pi_\lambda=[a,b]$ if $D_{\lambda}$ has one corner and $\pi_\lambda=[a,b,c,d]$ if $D_{\lambda}$ has two corners.
\end{Def}
\begin{Rem}
The numbers $a,b,c,d$ of the shape path can be visualized as follows.\par
\begin{figure}[ht!]
  \centering
    \includegraphics[width=250px]{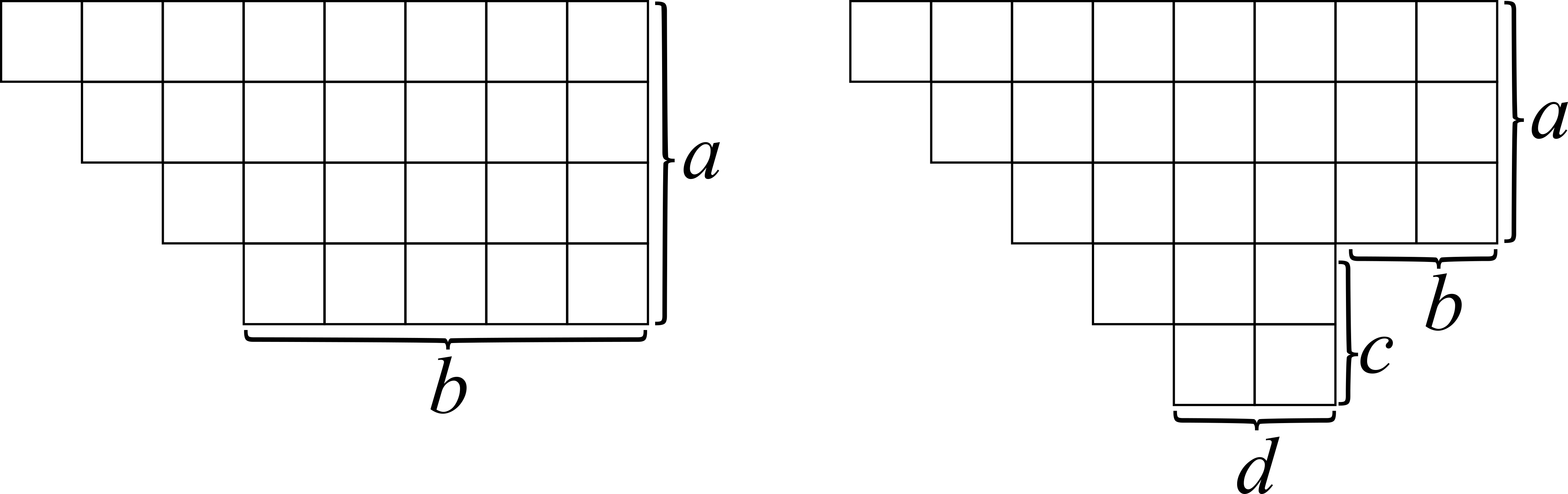}
\end{figure}
In particular, we see that 
for $\lambda=(8,7,6,5)$ we have $\pi_\lambda=[4,5]$,
and for $\lambda=(8,7,6,3,2)$ we have $\pi_\lambda=[3,2,2,2]$
\end{Rem}
\begin{Rem}
For a given $\lambda \in DP^{\leq 2}$ the cardinality of the border $B_{\lambda}$ can be derived from the shape path.
If $\lambda = [a,b]$ then $|B_{\lambda}| = a+b-1$.
If $\lambda = [a,b,c,d]$ then $|B_{\lambda}| = a+b+c+d-1$.
\end{Rem}

\begin{Lem}\label{shapepathbijection}
The map $DP^{\leq 2} \rightarrow \NN^4 \cup \NN^2: \lambda \mapsto \pi_\lambda$ is a bijection.
\end{Lem}
\begin{proof}
For a given $[a,b,c,d]$, its unique preimage is $\lambda = (a+b+c+d-1, a+b+c+d-2, \ldots, b+c+d+1, b+c+d, c+d-1, c+d-2, \ldots, d)$.
For a given $[a,b]$, its unique preimage is $\lambda = (a+b-1, a+b-2, \ldots, b)$.
Hence the map sending a partition in $DP^{\leq 2}$ to its shape path is bijective.
\end{proof}

\begin{Not}
From now on we will identify a partition $\lambda \in DP^{\leq 2}$ with at most two corners with its shape path $\pi_\lambda$.
\end{Not}

\begin{Lem}\label{notthreecornersmu}
Let $ \mu \in DP$, $\lambda \in DP^{\leq 2}$, and suppose  $\lambda \neq [a,b]$ where $b \leq 2$.
If $D_{\mu}$ has more than two corners then $Q_{\lambda/\mu}$ is not $Q$-multiplicity-free.
\end{Lem}
\begin{proof}
For each corner $(x,y)$ of $D_{\mu}$ except for the lowermost, there is a box $(x+1,z) \in D_{\lambda/\mu}$ such that $(x,z), (x+1,z-1) \notin D_{\lambda/\mu}$.
Also there is a box $(1,w) \in D_{\lambda/\mu}$ such that $(1,w-1) \notin D_{\lambda/\mu}$ and there is no box above because $(1,w)$ is in the first row.
After transposing this diagram orthogonally, the image of these boxes are corners of $D_{\lambda/\mu}^{ot}$.
The diagram $D_{\lambda/\mu}^{ot}$ has shape $D_{\alpha/\beta}$ where $\beta \neq \emptyset, (1)$ and $D_{\alpha}$ has more than two corners.
By Corollary \ref{notthreecornersco}, $Q_{\lambda/\mu}$ is not $Q$-multiplicity-free.
\end{proof}

\begin{Lem}\label{twocornersmu}
Let $ \mu \in DP$, $\lambda \in DP^{\leq 2}$, and suppose  $\lambda \neq
[a,b]$ where $b \leq 2$.
If $\mu = [w,x,y,z]$ where $z > 1$ then $Q_{\lambda/\mu}$ is not $Q$-multiplicity-free.
\end{Lem}
\begin{proof}
The leftmost box of the first row of $D_{\lambda/\mu}$, which is $(1,w+x+y+z)$, has no box to the left or above.
Also, the leftmost box of the $(w+1)^{\textrm{th}}$ row of $D_{\lambda/\mu}$, which is $(w+1,w+y+z)$, has no box to the left or above.
In addition, the leftmost box of the $(w+y+1)^{\textrm{th}}$ row of $D_{\lambda/\mu}$, which is $(w+y+1,w+y+1)$, has no box to the left or above.
After transposing this diagram orthogonally, the images of these boxes are corners of $D_{\lambda/\mu}^{ot}$.
Then the diagram $D_{\lambda/\mu}^{ot}$ has shape $D_{\alpha/\beta}$ where $\beta \neq \emptyset, (1)$ and $D_{\alpha}$ has more than two corners.
By Corollary \ref{notthreecornersco}, $Q_{\lambda/\mu}$ is not $Q$-multiplicity-free.
\end{proof}

\begin{Lem}\label{restrictions}
Suppose $\lambda = [a,b,c,d]$ and $\mu = [w,x]$ where $x > 1$ or $\mu = [w,x,y,1]$.
Then $Q_{\lambda/\mu}$ is not $Q$-multiplicity-free.
\end{Lem}
\begin{proof}
Let $k$ be such that $U_k(\lambda/\mu)$ has only one box in the diagonal $\{(s,t) \mid t-s = x-1\}$ for the case $\mu = [w,x]$ or in the diagonal $\{(s,t) \mid t-s = x+y\}$ for the case $\mu = [w,x,y,1]$.
Let this single remaining box be $(p,q)$.
Then $(p,q) \in P_k$ and also $(p-1,q), (p,q-1) \in P_k$.
Let $n = \ell(c(T_{\lambda/\mu}))$.\par
Case 1: $k = n$.\par
If $P_n$ is not a rotated hook, then by Lemma \ref{justhook}, $Q_{\lambda/\mu}$ is not $Q$-multiplicity-free.
If $P_n$ is a rotated $(l,m)$-hook where $l, m \geq 2$ then, since $\lambda = [a,b,c,d]$, there is some $j < n$ such that either the first box of $P_j$ is in a column to the right of the boxes of $P_n$ or the last box of $P_j$ is in a row below the boxes of $P_n$.
Let $j$ be maximal with respect to this condition.\par
We may assume that the first box of $P_j$ is in a column to the right of the boxes of $P_n$, otherwise $U_j(\lambda/\mu)$ is unshifted and we may consider $U_j(\lambda/\mu)^t$.
By Lemma \ref{ot2}, if $D_{\lambda/\mu}^{ot}$ has shape $D_{\alpha/\beta}$ then $T_{\alpha/\beta}^{(n)}$ is an $(m,l)$-hook where $l, m \geq 2$ and the diagram $U_j(\alpha/\beta)$ satisfies the conditions of Lemma \ref{fatcolumn}.
By Lemma \ref{parts}, it follows that $Q_{\lambda/\mu}$ is not $Q$-multiplicity-free.\par
Case 2: $k \neq n$.\par
If $U_{k+1}(\lambda/\mu)$ has at least two components then the last box of the second component can be filled with $(k+1)'$ or $k+1$ and, by Lemma \ref{checklist}, $Q_{\lambda/\mu}$ is not $Q$-multiplicity-free.
Thus, we may consider that all boxes of $U_{k+1}(\lambda/\mu)$ are either above or below the diagonal $\{(s,t) \mid t-s = x-1\}$ for the case $\mu = [w,x]$ or are above or below the diagonal $\{(s,t) \mid t-s = x+y\}$ for the case $\mu = [w,x,y,1]$.\par
Case 2.1: $P_n$ is an $(l,m)$-hook where $l, m \geq 2$.\par
Then either $U_k(\lambda/\mu)$ or $U_k(\lambda/\mu)^t$ satisfies the conditions of Lemma \ref{fatcolumn} and $Q_{\lambda/\mu}$ is not $Q$-multiplicity-free.\par
Case 2.2: only one corner is in $P_n$.\par
Let $(f,g)$ be this corner.
Then there is some $e$ such that there are two boxes of $P_e$ either in a row weakly below the $f^{\textrm{th}}$ row or in a column weakly to the right of $g^{\textrm{th}}$ column.
There is also some $h$ such that either the first box of $P_h$ is to the right of the $g^{\textrm{th}}$ column or the last box of $P_h$ is below the $f^{\textrm{th}}$ row.
Let $e,h$ be maximal with respect to these conditions.\par
By orthogonal transposition, transposition or rotation of $U_{\min\{e, h\}}(\lambda/\mu)$, we may assume that $h \leq e$ and that the first box of $P_h$ is to the right of the $g^{\textrm{th}}$ column.
By Lemma \ref{cornerfirst}, if $h = e$ then $Q_{\lambda/\mu}$ is not $Q$-multiplicity-free.
Hence, we assume $h < e$.\par
There is a box $(r,u) \in P_h$ in the diagonal $\{(s,t) \mid t-s = x-1\}$ for the case $\mu = [w,x]$ or in the diagonal $\{(t,s) \mid t-s = x+y\}$ for the case $\mu = [w,x,y,1]$.\par
We get a tableau $T$ if after the $(h-1)^{\textrm{th}}$ step of the algorithm of Definition \ref{Salmasian'salgorithm} we use $P'_h := P_h \setminus \{(r,u)\}$ instead of $P_h$.
By Corollary \ref{checklistco}, this tableau is $m$-amenable for $m \neq h+1$.
We have $T(r,u) = (h+1)'$ and $T(r-1,u-1) \neq h'$.
However, there is an $h$ with no $(h+1)$ below in the first box of $P_h$.
Thus, by Lemma \ref{checklist}, this tableau is amenable.\par
We get another tableau $T'$ with the same content if we set $T'(r-1,u) = h'$ and $T'(f,g) = T(f,g)$ for every other box $(f,g) \in D_{\lambda/\mu}$.
By Corollary \ref{checklistco}, this tableau is $m$-amenable for $m \neq h+1$.\par
We have $T'(r,u) = (h+1)'$ and $T'(r-1,u-1) \neq h'$.
However, we have $T'(r-1,u) = h'$.
In the $u^{\textrm{th}}$ column is an $h+1$ but no $h$.
However, there are $h$s with no $(h+1)$s below in the first box and in the last box of $P_h$.
Thus, by Lemma \ref{checklist}, this tableau is amenable.\par
By Lemma \ref{parts}, $Q_{\lambda/\mu}$ is not $Q$-multiplicity-free.
\end{proof}
\begin{Ex}
For $\lambda = [1,1,4,1]$ and $\mu = [1,1,1,1]$ we have ${\Yvcentermath1 T_{\lambda/\mu} = \young(:\meins 11,\meins 1\mzwei ,1\mzwei 2,:2\mdrei ,::3)}$.\\[1ex]
Then we obtain ${\Yvcentermath1 T = \young(:111,\meins \mzwei 2,1\mzwei \mdrei ,:2\mdrei ,::3)}, \vspace{1ex} {\Yvcentermath1 T' = \young(:\meins 11,\meins \mzwei 2,1\mzwei \mdrei ,:2\mdrei ,::3)}$.\\[1ex]
We have $Q_{(6,4,3,2,1)/(3,1)} = Q_{(6,4,2)}+2Q_{(5,4,3)}+Q_{(5,4,2,1)}$.
\end{Ex}

Now for $Q$-multiplicity-free skew Schur $Q$-functions $Q_{\lambda/\mu}$, for a given $\lambda$, the partition $\mu$ is restricted to certain families of partitions.
The following two lemmas and their corollaries restrict $\lambda$ and $\mu$ further until Proposition \ref{mf1} can be proved.

\begin{Lem}\label{fathook}
Let $\lambda = [a,b,c,1]$ and $\mu = [w,1]$.
If $a \geq 3$, $b \geq 3$, $c \geq 3$ and $4 \leq w \leq a+c-2$ then $Q_{\lambda/\mu}$ is not $Q$-multiplicity-free.
\end{Lem}
\begin{proof}
We will show that for case $a = 3$ and for case $w = a+c-2$ the statement holds.
Afterwards we will explain case $a > 3$ and $w < a+c-2$ by these two cases.\par
Case 1: $a = 3$.\par
Let $b \geq 3$, $c \geq 3$ and $4 \leq w \leq a+c-2$.
The lowermost box in the leftmost column of the diagram is $(w+1,w+1)$.
Since $w < a+c-1$, we have $(w,w+2) \in D_{\lambda/\mu}$.\par
We get a new tableau $T_1$ as follows:
In the algorithm of Definition \ref{Salmasian'salgorithm} use $P'_1 := P_1 \setminus \{(w+1,w+1)\}$, $P'_2 := P_2 \setminus \{(w+1,w+2), (w+2,w+2)\}$ and $P'_3 := P_3 \setminus \{(w,w+3),\linebreak
(w+1,w+3), (w+2,w+3), (w+3,w+3)\}$ (for $w = a+c-2$ this means $P'_3 = P_3$) instead of $P_1$, $P_2$ and $P_3$, respectively, and stop after the third step in the algorithm.
Then replace the entry $3$ in the last box of $P'_3$ with $3'$ and set $T_1(w+1,w+1) = 3$.
Afterwards fill the remaining boxes using the algorithm of Definition \ref{Salmasian'salgorithm} starting with $k = 4$.
By Corollary \ref{checklistco}, it is clear that $T_1$ is $m$-amenable for $m \neq 3, 4$.
There is a $3$ but no $2$ in the $(w+1)^{\textrm{th}}$ column.
However, there is a $2$ and no $3$ in the column of the last box of $P'_3$ and there is a $2$ and no $3$ in the column to the left of it.
Thus, by Lemma \ref{checklist}, this tableau is $3$-amenable.
In the $(w+2)^{\textrm{th}}$ column and possibly in the $(w+3)^{\textrm{th}}$ column, there are $4$s and no $3$s.
However, there are $3$s and no $4$s in the columns of the first two boxes of $P'_3$.
We have $T_1(w+1,w+2) = 4'$ and $T_1(w,w+1) \neq 3'$.
However, if $(y,z)$ is the third box of $P'_3$ then we either have $T_1(y,z) = 3$ and there is no $4$ in the $z^{\textrm{th}}$ column or if $w = a+c-2$ we have $T_1(y,z) = 3'$ and $T_1(y+1,z+1) \neq 4'$.
If $w < a+c-2$ then we have $T_1(w,w+3) = 4'$ and $T_1(w-1,w+2) \neq 3'$.
However, we have $T_1(w-1,w+3) = 3'$.
Thus, by Lemma \ref{checklist}, this tableau is $4$-amenable.\par
We get another tableau $T'_1$ of the same content if we set $T'_1(w+1,w+1) = 3$, $T'_1(w,w+2) = 2$ and $T_1'(r,s) = T_1(r,s)$ for every other box $(r,s) \in D_{\lambda/\mu}$.
It is easy to see that, by Corollary \ref{checklistco}, $T'_1$ is $m$-amenable for $m \neq 2, 3, 4$.
There is a $1$ with no $2$ below in the $(w+2)^{\textrm{th}}$ column.
Thus, by Lemma \ref{checklist}, $2$-amenability follows.
There is a $3$ with no $2$ above in the $(w+2)^{\textrm{th}}$ column.
However, there is a $2$ with no $3$ below in the column of the last box of $P_3$.
Thus, by Lemma \ref{checklist}, this tableau is $3$-amenable.
By Lemma \ref{checklist}, it is clear that $T'_1$ is $4$-amenable if $T_1$ is.\par
Case 2: $w = a+c-2$.\par
By Case 1, we may assume $a > 3$.
The lowermost box in the leftmost column of the diagram is $(w+1,w+1)$.
Since $w < a+c-1$, we have $(w,w+2) \in D_{\lambda/\mu}$.\par
Let $(y,z)$ be the last box of $P_3$.
We get a new tableau $T_2$ if we set $T_2(w+1,w+1) = 3$, $T_2(w,w+1) = 1$, $T_2(w,w+2) = 2$, $T_2(w+1,w+2) = 4$, $T_2(w+2,w+2) = 5$, $T_2(y,z) = 3'$, $T_2(y,z+1) = 4'$, for the case $P_5 \neq \emptyset$ set $T_2(y,z+2) = 5'$ (in this case $(y,z+2)$ is the last box of $P_5$), and set $T_2(r,s) = T_{\lambda/\mu}(r,s)$ for every other box $(r,s) \in D_{\lambda/\mu}$.\par
By Corollary \ref{checklistco}, $T_2$ is $m$-amenable for $m \neq 3, 4, 5$.
There is a $3$ and no $2$ in the $(w+1)^{\textrm{th}}$ column.
However, there is are $2$s and no $3$s in the $z^{\textrm{th}}$ and in the $(w+2)^{\textrm{th}}$ column.
Thus, by Lemma \ref{checklist}, this tableau is $3$-amenable.
There is a $4$ with no $3$ above in the $(w+2)^{\textrm{th}}$ column.
However, there are $3$s and no $4$s in the $(w+1)^{\textrm{th}}$ column and in the $(z+1)^{\textrm{th}}$ column.
Thus, by Lemma \ref{checklist}, $4$-amenability follows.
The $5$-amenability is clear for $P_5 = \emptyset$.
If $P_5 \neq \emptyset$ then there is a $4$ and no $5$ in the $(z+2)^{\textrm{th}}$ column.
Thus, by Lemma \ref{checklist}, this tableau is $5$-amenable.\par
We get another tableau $T'_2$ of the same content if we set $T'_2(w+1,w+1) = 2$, $T'_2(w,w+2) = 3$ and $T'_2(r,s) = T_2(r,s)$ for every other box $(r,s) \in D_{\lambda/\mu}$.
By Corollary \ref{checklistco}, $T'_2$ is $m$-amenable for $m \neq 2, 3, 4$.
There is a $1$ and no $2$ in the $(w+2)^{\textrm{th}}$ column.
Thus, by Lemma \ref{checklist}, $2$-amenability follows.
There is a $3$ and no $2$ in the $(w+2)^{\textrm{th}}$ column.
However, there is a $2$ with no $3$ below in the $z^{\textrm{th}}$ column.
Thus, by Lemma \ref{checklist}, this tableau is $3$-amenable.
By Lemma \ref{checklist}, it is clear that $T'_2$ is $4$-amenable if $T_2$ is.\par
Case 3: $a > 3$ and $w < a+c-2$.\par
The diagram $U_2(\lambda/\mu)$ has shape $D_{\lambda'/\mu}$ where $\lambda' = [a',b,c,1]$ where $a' = a-1$.
Either we have $a' = a-1 = 3$ or $w = a'+c-2$ or else there is some $j$ such that $U_j(\lambda/\mu)$ has shape $D_{\lambda''/\mu}$ where $\lambda'' = [a'',b,c,1]$ where $a'' = a-j$ such that either $a'' = 3$ or $w = a''+c-2$.
Then, by Case 1 and Case 2, we find two different amenable tableaux of the same content and, by Lemma \ref{parts}, $Q_{\lambda/\mu}$ is not $Q$-multiplicity-free.
\end{proof}
\begin{Ex}
For $\lambda = [3,3,6,1]$ and $\mu = [5,1]$ the tableaux are
$$T_1 = {\Yvcentermath1 \young(\meins 111111,\meins \mzwei 22222,\meins \mzwei \mdrei 3333,\meins \mzwei \mdrei \mvier ,12\mvier 4,3\mvier 4\mfuenf ,:4\mfuenf 5,::5\msechs,:::6)}, \hspace{1ex} T'_1 = {\Yvcentermath1 \young(\meins 111111,\meins \mzwei 22222,\meins \mzwei \mdrei 3333,\meins \mzwei \mdrei \mvier ,13\mvier 4,2\mvier 4\mfuenf ,:4\mfuenf 5,::5\msechs,:::6)}.$$
For $\lambda = [4,5,3,1]$ and $\mu = [5,1]$ the tableaux are
$$T_2 = {\Yvcentermath1 \young(\meins 111111,\meins \mzwei 22222,\meins \mzwei \mdrei 3333,\meins \mzwei \mdrei \mvier 444,12,34,:5)}, \hspace{1ex} T'_2 = {\Yvcentermath1 \young(\meins 111111,\meins \mzwei 22222,\meins \mzwei \mdrei 3333,\meins \mzwei \mdrei \mvier 444,13,24,:5)}.$$
\end{Ex}
\begin{Co}\label{fathookco}
Let $\lambda = [a,b]$ and $\nu = [w,x]$.
If $w \geq 3$, $x \geq 4$, $a \geq w+2$, $b \geq 5$ and $a+b-w-x \geq 3$ then $Q_{\lambda/\mu}$ is not multiplicity-free.
\end{Co}
\begin{proof}
If $\lambda, \mu$ satisfy these properties then the diagram $D^{ot}_{\lambda/\mu}$ is equal to $D_{\alpha/\beta}$ where $\alpha = [a',b',c',1]$ and $\beta = [w',1]$.
Then $b' = w \geq 3$ and $c' = x-1 \geq 3$.
The number $a'$ is the number of boxes of the first row of $D_{\lambda/\mu}$ and can be calculated by $a' = \lambda_1 - \mu_1 = |B_{\lambda}| - |B_{\mu}| = a+b-w-x \geq 3$.
Since $a \geq w+2$, we have $a-w-2 \geq 0$ and, hence, $b \leq a+b-w-2$.
Then we get $4 \leq b-1 = w' = b-1 \leq a+b-w-2-1 = a+b-w-x+x-1-2 = a'+c'-2$.
By Lemma \ref{fathook}, $Q_{D_{\lambda/\mu}^{ot}}$ is not $Q$-multiplicity-free and, thus, $Q_{\lambda/\mu}$ is not $Q$-multiplicity-free.
\end{proof}
\begin{Ex}
The smallest diagram satisfying the properties of Corollary \ref{fathookco} is $D_{(9,8,7,6,5)/(6,5,4)}$.\par
We have $Q_{(9,8,7,6,5)/(6,5,4)} = Q_{(9,8,3)}+Q_{(9,7,4)}+Q_{(9,7,3,1)}+Q_{(9,6,4,1)}+Q_{(9,6,3,2)}+Q_{(9,5,4,2)}+Q_{(8,7,5)}+Q_{(8,7,4,1)}+Q_{(8,7,3,2)}+Q_{(8,6,5,1)}+2Q_{(8,6,4,2)}+Q_{(8,6,3,2,1)}+Q_{(8,5,4,3)}+Q_{(8,5,4,2,1)}+Q_{(7,6,5,2)}+Q_{(7,6,4,3)}+Q_{(7,6,4,2,1)}+Q_{(7,5,4,3,1)}$.
\end{Ex}

\begin{Lem}\label{fathook1}
Let $\lambda = [a,b,c,d]$ and $\mu = [w,1]$.
If $a,b,c,d \geq 2$ and $3 \leq w \leq a+c-1$ then $Q_{\lambda/\mu}$ is not $Q$-multiplicity-free.
\end{Lem}
\begin{proof}
Let $n = \ell(c(T_{\lambda/\mu}))$.
First, we assume $w = a+c-1$ and prove the result when $a = 2$ or $d = 2$.
Then, we show that the case where $a, d \geq 3$ can be explained by
the cases where $a = 2$ or $d = 2$.
Afterwards, we tackle the case $w < a+c-1$ using the case $w = a+c-1$ where we first prove the subcase $d = 2$ and then show how to add boxes with entries to obtain diagrams such that $d > 2$.\par
Case 1: $w = a+c-1$ and $2 \in \{a, d\}$.\par
We may assume $a = 2$, otherwise we transpose the diagram.
If $d = 2$ then $P_n$ is a $(b+1,c+1)$-hook and, by Lemma \ref{threethreehook}, which in this case forbids to have a box directly to the left of the last box of $P_n$, $Q_{\lambda/\mu}$ is not $Q$-multiplicity-free.
Thus, we may now assume $d \geq 3$.\par
The box $(w+1,w+1)$ is the last box of $P_1$.
We get a new tableau $T_1$ if we set $T_1(w,w+1) = 1$, $T_1(w+1,w+1) = 3$, $T_1(w,w+2) = 2$, $T_1(w+1,w+2) = 3$, $T_1(w,w+3) = 3$, $T_1(w+1,w+3) = 4$ and $T_1(r,s) = T_{\lambda/\mu}(r,s)$ for every other box $(r,s) \in D_{\lambda/\mu}$.\par
By Corollary \ref{checklistco}, $T_1$ is $m$-amenable for $m \neq 3$.
There is a $3$ and no $2$ in the $(w+1)^{\textrm{th}}$ column.
However, there are $2$s and no $3$s in the columns of the first two boxes of $P_2$.
Thus, by Lemma \ref{checklist}, $T_1$ is amenable.\par
We get another tableau $T'_1$ if we set $T'_1(w+1,w+1) = 2$, $T'_1(w,w+2) = 3'$ and $T'_1(r,s) = T_1(r,s)$ for every other box $(r,s) \in D_{\lambda/\mu}$.\par
By Corollary \ref{checklistco}, $T'_1$ is $m$-amenable for $m \neq 2, 3$.
In the $(w+2)^{\textrm{th}}$ column is a $1$ with no $2$ below.
Thus, by Corollary \ref{checklistco}, $2$-amenability follows.
We have $T'_1(w,w+2) = 3'$ and $T'_1(w-1,w+1) \neq 2'$ and there is a $3$ and no $2$ in the $(w+2)^{\textrm{th}}$ column.
However, there are two $2$s and no $3$s in the columns of the first two boxes of $P_2$.
Thus, by Lemma \ref{checklist}, $3$-amenability follows.
It is clear that $T'_1$ has the same content as $T_1$.
Hence, $Q_{\lambda/\mu}$ is not $Q$-multiplicity-free.\par
Case 2: $w = a+c-1$ and $a, d \geq 3$.\par
The diagram $U_2(\lambda/\mu)$ has shape $D_{\alpha/\beta}$ where $\alpha = [a',b,c,d']$ and $\beta = [a'+c-1,1]$, and $a' = a-1$ and $d' = d-1$.
If $a' = 2$ or $d' = 2$ then Case 1 proves the statement.
Otherwise, there is some $j$ such that $U_j(\lambda/\mu)$ has shape $D_{\alpha'/\beta'}$ where $\alpha = [a'',b,c,d'']$ and $\beta = [a''+c-1,1]$, and $a'' = 2$ or $d'' = 2$.
By Lemma \ref{parts} and Case 1, $Q_{\lambda/\mu}$ is not $Q$-multiplicity-free.\par
Case 3: $3 \leq w < a+c-1$.\par
Assume $a > 2$.
Let $(x,y)$ be the lower corner.
Since $w < a+c-1$, the last box of $P_1$ is not in the $x^{\textrm{th}}$ row.
Then the diagram $U_2(\lambda/\mu)$ has shape $D_{\lambda/\mu}$ where $\lambda' = [a-1,b,c,d]$ and $\mu' = [w,1]$.
Then there is some $j$ such that $U_j(\lambda/\mu)$ has shape $D_{\alpha'/\beta'}$ where either $\alpha' = [2,b,c,d]$ and $\beta' = [w,1]$ or where $\alpha' = [e,b,c,2]$ and $\beta' = [w',1]$ where $a > e \geq 3$ and $w' = e+c-1$.
In the latter case the transpose of the diagram is covered in Case 2.
Thus, it suffices to consider the case $D_{\alpha/\beta}$ where $\alpha = [2,b,c,d]$ and $\beta = [w,1]$ and $3 \leq w < 2+c-1 = c+1$.\par
Case 3.1: $d = 2$.\par
The box $(w+1,w+1)$ is the last box of $P_1$.
We get a new tableau $T_2$ as follows:
In the algorithm of Definition \ref{Salmasian'salgorithm} use $P'_1 := P_1 \setminus \{(w+1,w+1)\}$ and $P'_2 := P_2 \setminus \{(w+1,w+2),(w+2,w+2)\}$ instead of $P_1$ and $P_2$, respectively.
By Corollary \ref{checklistco}, $T_2$ is $m$-amenable for $m \neq 3$.
There is a $3$ and no $2$ in the $(w+1)^{\textrm{th}}$ column.
However, there are $2$s and no $3$s in the columns of the first two boxes of $P_2$.
Thus, by Lemma \ref{checklist}, $3$-amenability follows.\par
We get another tableau $T'_2$ as follows:
\begin{itemize}
	\item Set $T'_2(r,s) = T_2(r,s)$ for every $(r,s) \in P'_1 \cup (P'_2 \setminus \{(w,w+2)\})$ where $P'_1$ and $P'_2$ as above.
	\item Set $T'_2(w+1,w+1) = 2$.
	\item Fill the remaining boxes using the algorithm of Definition \ref{Salmasian'salgorithm} starting with $k = 3$.
\end{itemize}
By Corollary \ref{checklistco}, $T'_2$ is $m$-amenable for $m \neq 2, 3$.
There is a $1$ and no $2$ in the $(w+2)^{\textrm{th}}$ column.
Thus, by Corollary \ref{checklistco}, $2$-amenability follows.
There is a $3$ and no $2$ in the $(w+2)^{\textrm{th}}$ column.
However, there is a $2$ and no $3$ in the column of the first box of $P_2$.
We have $T'_2(w+1,w+2) = 3'$ and $T'_2(w,w+1) \neq 2'$.
However, there is a $2$ and no $3$ in the column of the second box of $P_2$.
We have $T'_2(w,w+2) = 3'$ and $T'_2(w-1,w+1) \neq 2'$.
However, we have $T'_2(w-1,w+2) = 2'$.
Thus, by Lemma \ref{checklist}, $3$-amenability follows.\par
We have $|T_2(w+1+j,w+1+j)| = j+3$ and $|T_2(w+j,w+2+j)| = j+2$ for $0 \leq j \leq n-2$ and we have $|T'_2(w+1+j,w+1+j)| = j+2$ and $|T'_2(w+j,w+2+j)| = j+3$ for $0 \leq j \leq n-2$.
The entries of the other boxes in $T_2$ and $T'_2$ can only differ by markings.
Thus, $T'_2$ has the same content as $T_2$.\par
Case 3.2: $d > 2$.\par
Let $(x,y)$ be the lower corner.
We get two tableaux $\tilde{T}_2$ and $\tilde{T}'_2$ of shape $D_{\alpha/\beta}$ where $\alpha = [2,b,c,d]$ and $\beta = [w,1]$ if we take the two tableaux from Case 3.1 of shape $D_{\alpha'/\beta'}$ where $\alpha' = [2,b,c,2]$ and $\beta' = [w,1]$ and add $d-2$ columns using the following algorithm:
\begin{enumerate}
	\item Set $\tilde{T}_2(e,f) = T_2(e,f)$ and $\tilde{T}'_2(e,f) = T'_2(e,f)$ for all $f \leq y$ and for all $e$ such that $(e,f) \in D_{\lambda/\mu}$.
	\item Set $\tilde{T}_2(p,q) = T_2(p,q-d+2)$ and $\tilde{T}'_2(p,q) = T'_2(p,q-d+2)$ for all $q > y$ and for all $p$ such that $(p,q) \in D_{\lambda/\mu}$.
	\item For $1 \leq j \leq n$ set $\tilde{T}_2(j,y+1) = \tilde{T}'_2(j,y+1) = j$.
	\item For $n+1 \leq r \leq x-2$ set $\tilde{T}_2(r,y+1) = \tilde{T}'_2(r,y+1) = (n+1)'$.
	\item Set $\tilde{T}_2(x-1,y+1) = \tilde{T}'_2(x-1,y+1) = n+1$ and set $\tilde{T}_2(x,y+1) = \tilde{T}'_2(x,y+1) = n+2$.
	\item Do the following algorithm:
	\begin{enumerate}[(i)]
		\item Set $i = y+2$:
		\item Scan the $(i-1)^{\textrm{th}}$ column of $\tilde{T}_2$ from top to bottom and find the uppermost marked letter, $z$ say. If there is no marked letter in the $(i-1)^{\textrm{th}}$ column then set $z = 2+c$.
		\item For $1 \leq r \leq |z|$ set $\tilde{T}_2(r,i) = \tilde{T}'_2(r,i) = r$.
		\item For $|z|+1 \leq s \leq 2+c$ set $\tilde{T}_2(s,i) = \tilde{T}'_2(s,i) = t+1$ if $\tilde{T}_2(s-1,i) = \tilde{T}_2(s-1,i) = t$ or else set $\tilde{T}_2(s,i) = \tilde{T}'_2(s,i) = (t+1)'$ if $\tilde{T}_2(s-1,i) = \tilde{T}_2(s-1,i) = t'$.
		\item Increment $i$.
		\item If $i \leq d-2$ go to (ii) or else stop.
	\end{enumerate}
\end{enumerate}
It is easy to see that these tableaux are amenable if the tableaux for $d = 3$ are amenable.
By definition of the algorithm, if we have $T_2(u,y+1) = T'_2(u,y+1) = (n+1)'$ then $T_2(u-1,y) = T'_2(u-1,y) = n'$.
Hence, by Lemma \ref{checklist}, these tableaux are amenable.\par
For $d > 3$, since the $(y+1)^{\textrm{th}}$ column has the same entries in both tableaux, the algorithm fills the other $d-3$ columns in the same amenable way.
Clearly, the contents of $\tilde{T}_2$ and $\tilde{T}'_2$ are equal.
\end{proof}
\begin{Ex}
For $\lambda = [2,2,3,5]$ and $\mu = [4,1]$ we have
$${\Yvcentermath1 T_1 = \young(\meins 111111,\meins \mzwei 22222,\meins \mzwei \mdrei 33,123\mvier 4,33445), \hspace{1ex} T'_1 = \young(\meins 111111,\meins \mzwei 22222,\meins \mzwei \mdrei 33,1\mdrei 3\mvier 4,23445)}.$$
For $\lambda = [2,2,6,4]$ and $\mu = [5,1]$ we first take the tableaux for $\lambda' = [2,2,6,2]$ and $\mu' = [5,1]$:
$${\Yvcentermath1 T_2 = \young(\meins 11111,\meins \mzwei 2222,\meins \mzwei \mdrei 3,\meins \mzwei \mdrei \mvier ,12\mdrei \mvier ,333\mvier ,:444,::55), \hspace{1ex} T'_2 = \young(\meins 11111,\meins \mzwei 2222,\meins \mzwei \mdrei 3,\meins \mzwei \mdrei \mvier ,1\mdrei 3\mvier ,2\mdrei \mvier 4,:3\mvier \mfuenf,::45)}.$$
Then we add two columns using the algorithm of Lemma \ref{fathook1}:
$${\Yvcentermath1 \tilde{T}_2 = \young(\meins 1111111,\meins \mzwei 222222,\meins \mzwei \mdrei 333,\meins \mzwei \mdrei \mvier 44,12\mdrei \mvier \mfuenf 5,333\mvier \mfuenf \msechs,:44456,::5567), \hspace{1ex} \tilde{T}'_2 = \young(\meins 1111111,\meins \mzwei 222222,\meins \mzwei \mdrei 333,\meins \mzwei \mdrei \mvier 44,1\mdrei 3\mvier \mfuenf 5,2\mdrei \mvier 4\mfuenf \msechs,:3\mvier \mfuenf56,::4567)}.$$
\end{Ex}
\begin{Co}\label{fathook1co}
Let $\lambda = [a,b]$ and $\mu = [w,x,y,1]$.
If $w \geq 2$, $x \geq 2$, $b \geq 4$ and $a+b-1-w-x-y \geq 2$ then $Q_{\lambda/\mu}$ is not $Q$-multiplicity-free.
\end{Co}
\begin{proof}
If $\lambda, \mu$ satisfy these properties then the diagram $D^{ot}_{\lambda/\mu}$ is equal to $D_{\alpha/\beta}$ where $\alpha = [a',b',c',d']$ and $\beta = [w',1]$ where $b' = w \geq 2$, $c' = x \geq 2$, $d' = y+1 \geq 2$ and additionally $a'+c'-1 \geq w' = b-1 \geq 3$.
The number $a'$ is the number of boxes of the first row of $D_{\lambda/\mu}$ and can be calculated by $a' = \lambda_1 - \mu_1 = |B_{\lambda}| - |B_{\mu}| = a+b-1-w-x-y \geq 2$.
By Lemma \ref{fathook1}, $Q_{D_{\lambda/\mu}^{ot}}$ is not $Q$-multiplicity-free and, thus, $Q_{\lambda/\mu}$ is not $Q$-multiplicity-free.
\end{proof}

As we will see soon we have already determined all non-$Q$-multiplicity-free skew Schur $Q$-functions.
The following proposition gives a list of all skew Schur $Q$-functions that are possibly $Q$-multiplicity-free.
This is half of the classification of $Q$-multiplicity-free skew Schur $Q$-functions.

\begin{Prop}\label{mf1}
Let $\lambda, \mu \in DP$ such that $D_{\lambda/\mu}$ is basic.
Let $a,b,c,d,w,x,y \in \NN$.
If $Q_{\lambda/\mu}$ is $Q$-multiplicity-free then $\lambda$ and $\mu$ satisfy one of the following conditions:
\begin{enumerate}[(i)]
	\item $\lambda$ is arbitrary and $\mu \in \{\emptyset, (1)\}$,\label{a}
	\item $\lambda = [a,b]$ where $b \in \{1,2\}$ and $\mu$ is arbitrary,\label{b}
	\item $\lambda = [a,b]$ and $\mu = [w,x,y,1]$ where $a+b-w-x-y-1 = 1$ or $w = 1$ or $x = 1$ or $b \leq 3$,\label{c}
	\item $\lambda = [a,b,c,d]$ where $d \neq 1$ and $\mu = [w,1]$ where $1 \in \{a,b,c\}$ or $w \leq 2$,\label{d}
	\item $\lambda = [a,b,c,1]$ and $\mu = [w,1]$ where $a \leq 2$ or $b \leq 2$ or $c \leq 2$ or $w \leq 3$ or $w = a+c-1$.\label{e}
	\item $\lambda = [a,b]$ and $\mu = [w,x]$ where $2 \leq b \leq 4$ or $w \leq 2$ or $x \leq 3$ or $a = w+1$ or $a+b-w-x \leq 2$.\label{f}
\end{enumerate}
Some of these cases overlap.\par
The cases (iii) - (vi) are depicted as diagrams in the remark after the proof of this proposition.\par
We want to note that Case (\ref{a}) is the orthogonal transposition of Case (\ref{b}).
Also, Case (\ref{c}) is the orthogonal transposition of Case (\ref{d}).
Case (\ref{e}) is the orthogonal transposition of Case (\ref{f}) for $x > 1$.
The orthogonal transposition of Case (\ref{f}) for $x = 1$ is also covered in Case (\ref{f}).
\end{Prop}
\begin{proof}
If $\mu = \emptyset, (1)$ we have no restrictions for $\lambda$.
We also have no restrictions for $\mu$ if $\lambda = [a,b]$ where $b \in \{1,2\}$.\par
Now consider $\mu \notin \{\emptyset, (1)\}$ and if $\lambda = [a,b]$ then $b \geq 3$.
Then by Lemma \ref{notthreecornersmu}, Lemma \ref{twocornersmu} and Lemma \ref{restrictions}, if $Q_{\lambda/\mu}$ is $Q$-multiplicity-free then $\lambda$ and $\mu$ satisfy one of the following cases:
\begin{itemize}
	\item $\lambda = [a,b]$ and $\mu = [w,x]$
	\item $\lambda = [a,b]$ and $\mu = [w,x,y,1]$
	\item $\lambda = [a,b,c,d]$ and $\mu = [w,1]$
\end{itemize}
for some $a,b,c,d,w,x,y \in \NN$.
Note that in the last case if $w \geq a+c$ then $\ell(\mu) \geq \ell(\lambda)$ and the diagram $D_{\lambda/\mu}$ is either not defined or is not basic since it has an empty column.
Hence, we will only consider $w \leq a+c-1$.\\
By Corollary \ref{fathookco}, for the case $\lambda = [a,b]$ and $\mu = [w,x]$, we have the restriction $b \leq 4$ or $w \leq 2$ or $x \leq 3$ or $a = w+1$ or $a+b-w-x \leq 2$.\\
By Corollary \ref{fathook1co}, for the case $\lambda = [a,b]$ and $\mu = [w,x,y,1]$, we have the restriction $w = 1$ or $x = 1$ or $b \leq 3$ or $a+b-w-x-y-1 = 1$.\\
By Lemma \ref{fathook1}, for the case $\lambda = [a,b,c,d]$ where $d \neq 1$ and $\mu = [w,1]$, we have the restriction $1 \in \{a,b,c\}$ or $w \leq 2$.\\
By Lemma \ref{fathook}, for the case $\lambda = [a,b,c,1]$ and $\mu = [w,1]$, we have the restriction $a \leq 2$ or $b \leq 2$ or $c \leq 2$ or $w \leq 3$ or $w = a+c-1$.
\end{proof}

\begin{Rem}
The following shows the diagrams in cases (iii) - (vi) of Proposition \ref{mf1};  here all boxes of a diagram belong to $\lambda$ and the gray boxes
belong to the diagram of $\mu$:
{\flushleft Case \ref{mf1} (\ref{c}):\\}
{\centering \includegraphics[width=3.5cm]{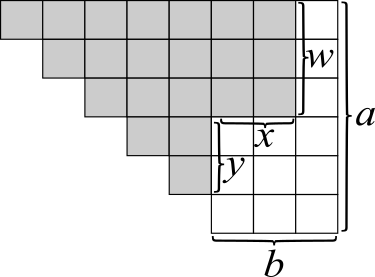}\\
$a+b-w-x-y-1 = 1$ or $w = 1$ or $x = 1$ or $b \leq 3$.\\}
{\flushleft Case \ref{mf1} (\ref{d}):\\}
{\centering \includegraphics[width=3.5cm]{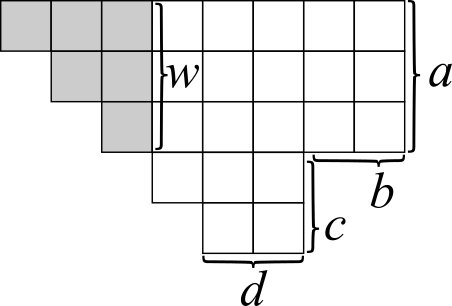}\\
If $d \geq 2$ then $1 \in \{a,b,c\}$ or $w \leq 2$.\\}
{\flushleft Case \ref{mf1} (\ref{e}):\\}
{\centering \includegraphics[width=3.5cm]{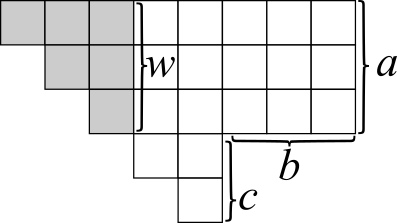}\\
$a \leq 2$ or $b \leq 2$ or $c \leq 2$ or $w \leq 3$ or $w = a+c-1$.\\}
{\flushleft Case \ref{mf1} (\ref{f}):\\}
{\centering \includegraphics[width=3.5cm]{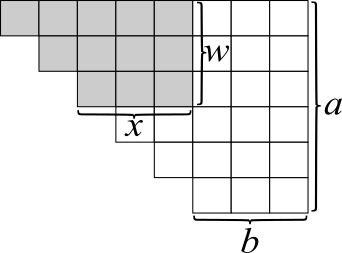}\\
$2 \leq b \leq 4$ or $w \leq 2$ or $x \leq 3$ or $a = w+1$ or $a+b-w-x \leq 2$.\\}
\end{Rem}

To show that the list in Proposition \ref{mf1} gives the classification of $Q$-multiplicity-free skew Schur $Q$-functions we have to prove the $Q$-multiplicity-freeness in each of these cases.
We will do this in the following until we are able to state the
classification result as Theorem \ref{listmf}.

The next lemma shows the $Q$-multiplicity-freeness of \ref{mf1} (\ref{a}).

\begin{Lem}\label{casea}
If $\lambda$ is arbitrary and $\mu = \emptyset$ then $Q_{\lambda/\mu} = Q_{\lambda}$ and, thus, $Q_{\lambda/\mu}$ is $Q$-multiplicity-free.\par
If $\lambda$ is arbitrary and $\mu = (1)$ then
$$Q_{\lambda/\mu} = \sum_{\nu \in E_{\lambda}}{Q_{\nu}},$$
where $E_{\lambda}$ is the set from Definition \ref{E}.
In particular, $Q_{\lambda/\mu}$ is $Q$-multiplicity-free.
\end{Lem}
\begin{proof}
For $\mu = \emptyset$ we have $Q_{\lambda/\emptyset} = Q_{\lambda}$.
Thus, $f^{\lambda}_{\emptyset \lambda} = 1$ and $f^{\lambda}_{\emptyset \nu} = 0$ for $\nu \neq \lambda$.
Hence, $Q_{\lambda/\emptyset}$ is $Q$-multiplicity-free.\par
The case $\mu = (1)$ is a case of Proposition \ref{lambda/n}.
\end{proof}
\begin{Ex}
Since $E_{(8,6,5,1)} = \{(7,6,5,1),(8,6,4,1),(8,6,5)\}$ we have
$$Q_{(8,6,5,1)/(1)} = Q_{(7,6,5,1)}+Q_{(8,6,4,1)}+Q_{(8,6,5)}.$$
\end{Ex}

Before showing the $Q$-multiplicity-freeness of \ref{mf1} (\ref{b}) we will define a notation in order to describe the decomposition for a subcase of \ref{mf1} (\ref{b}).

\begin{Def}
Let $\lambda = (\lambda_1, \lambda_2, \ldots, \lambda_{\ell(\lambda)}) \in DP$.
Let $\mu = (\lambda_{i_1}, \lambda_{i_2}, \ldots, \lambda_{i_{\ell(\mu)}})$ such that $\{i_1, i_2, \ldots, i_{\ell(\mu)}\} \subseteq \{1, 2, \ldots, \ell(\lambda)\}$.
Then $\lambda \setminus \mu$ is defined as the partition obtained by removing the parts of $\mu$ from $\lambda$.
\end{Def}
\begin{Ex}
For $\lambda = (9,7,5,4,3,1)$ and $\mu = (5,3,1)$ we obtain $\lambda \setminus \mu = (9,7,4)$.
\end{Ex}

\begin{Lem}\label{caseb}
If $\lambda = [a,b]$ where $b \in \{1,2\}$ and $\mu$ is arbitrary then $Q_{\lambda/\mu}$ is $Q$-multiplicity-free.
In particular, if $\lambda = [a,1]$ then $Q_{\lambda/\mu} = Q_{\lambda \setminus \mu}$.
\end{Lem}
\begin{proof}
Case 1: $b = 2$.\\
Then $D_{\lambda/\mu}^{ot}$ has shape $D_{\alpha/(1)}$ for an $\alpha \in DP$.
So, by Lemma \ref{ot}, $Q_{\lambda/\mu} = Q_{D_{\lambda/\mu}^{ot}} = Q_{\alpha/(1)}$.
By Lemma \ref{casea}, $Q_{\alpha/(1)}$ is $Q$-multiplicity-free.\\
Case 2: $b = 1$.\\
The $Q$-multiplicity-freeness as well as the decomposition of $Q_{\lambda/\mu}$ is is proved in \cite[Lemma 4.19]{Schure}.
\end{proof}
\begin{Rem}
In \cite{Salmasian} there is already a proof for the statement that $Q_{[a,1]/\mu} = Q_{\nu}$ for some $\nu$.
Using Lemma \ref{ot}, this statement follows immediately and Lemma \ref{caseb} helps to obtain this $\nu$.
\end{Rem}

We postpone to prove the $Q$-multiplicity-freeness of \ref{mf1} (\ref{c}).
We will first show the $Q$-multiplicity-freeness of \ref{mf1} (\ref{d}) and then prove that \ref{mf1} (\ref{c}) is the orthogonally transposed version of \ref{mf1} (\ref{d}).

\begin{Lem}\label{uniquefilling}
Let $D$ be a basic diagram of shape $D_{\lambda/[s,1]}$ for some $s$.
If the first $a$ rows of $D$ form a diagram $D_{\alpha/\beta}$ where $\alpha = [a,b]$ and $\beta = [w,1]$ then the filling of the boxes of the first $a$ rows of $D$ in any amenable tableau $T$ of $D$ is unique up to marks.
\end{Lem}
\begin{proof}
Let the diagram be shifted such that the uppermost leftmost box is $(1,1)$, the uppermost rightmost box is $(1,a+b-w-1)$ and the lowermost rightmost box is the box $(a,a+b-w-1)$.
Let $T$ be an amenable tableau of $D$.\par
Case 1: $w = a-1$.\par
Then the uppermost leftmost box is $(1,1)$, the uppermost rightmost box is $(1,b)$, the lowermost leftmost box is $(a,1)$ and the lowermost rightmost box is $(a,b)$.
Let $T_{(j)}$ be the subtableau of $T$ consisting of the boxes with their entries of the first $j$ rows.
We need to show that $T_{(j)} \cap T^{(i)}$ is a $(j+1-i,b+1-i)$-hook at $(i,i)$ for $1 \leq i \leq \min\{b,j\}$ where $T^{(i)}$ is as in Definition \ref{boxeswithi}.\par
Case 1.1: $T_{(j)} \cap T^{(1)}$ is not a $(j,b)$-hook at $(1,1)$ but $T_{(j-1)} \cap T^{(i)}$ is a $(j-i,b+1-i)$-hook at $(i,i)$ for all $1 \leq i \leq \min\{b,j-1\}$ for some $j$.\par
Then we have $T(j,1) > 1$.
Let $t = T(j,1)$.
For $t \in \{j', j\}$, by Lemma \ref{firstrows}, all boxes in the $j^{\textrm{th}}$ row are then filled with entries from $\{j', j\}$.
The remark after Definition \ref{amenable} implies that $c^{(u)}(T)_j = b \geq c^{(u)}(T)_{j-1}$; a contradiction to Lemma \ref{unmarked}.Thus, we have $1 < t < j'$.
Then the last box of $T_{(j-1)} \cap T^{(t)}$ contains a $|t|'$, for otherwise, by the remark after Definition \ref{amenable}, we have at least as many $|t|$s as $(|t|-1)$s, which contradicts Lemma \ref{unmarked}.
We have $|T(j,2)| > |t|$.
Otherwise, we would have at least as many $|t|$s as $(|t|-1)$s, which contradicts Lemma \ref{unmarked}.\par
Repeating this argument, we get $|T(j,s)| > |T(j,s-1)|$ for $2 \leq s \leq r$ where $r$ is such that $T(j,r+1)$ is the leftmost box with an entry that does not appear in the first $(j-1)^{\textrm{th}}$ rows.\par
By Lemma \ref{firstrows}, $T(j,r+1) \in \{j', j\}$ and, hence $T(j,k) \in \{j', j\}$ for $r+1 \leq k \leq b$.
If $T(j-1,r+1) \notin \{(j-1)', j-1\}$ then the remark after Definition \ref{amenable} implies that $c^{(u)}(T)_j > c^{(u)}(T)_{j-1}$; a contradiction to Lemma \ref{unmarked}.
Hence, we have $T(j-1,r+1) \in \{(j-1)', j-1\}$.
If $T(j,r) \notin \{(j-1)', j-1\}$ then, again, the remark after Definition \ref{amenable} implies that $c^{(u)}(T)_j \geq c^{(u)}(T)_{j-1}$; a contradiction to Lemma \ref{unmarked}.
If $T(j,r) \in \{(j-1)', j-1\}$ then $T(j-1,r+1) = (j-1)'$.
Let $(j,r+1)$ be the box of the $l^{\textrm{th}}$ letter of the reading word $w(T)$.
Then $m_{j-1}(n-l) = m_j(n-l)$ and $w(T)_l \in \{j', j\}$, contradicting Definition \ref{amenable} $a)$.\par
Case 1.2: $T_{(j)} \cap T^{(v)}$ is not a $(j+1-v,b+1-v)$-hook at $(v,v)$ but $T_{(j-1)} \cap T^{(i)}$ is a $(j+1-i,b-i)$-hook at $(i,i)$ for all $1 \leq i \leq \min\{b-1,j\}$ for some $j$ and some minimal $v \leq j-1$.\par
Let $j$ be minimal with respect to this property.
By Case 1.1, we may assume that $v > 1$.
Let $v$ be minimal with respect to this property.
Then we may take $T_{(j)}$, remove $P_1, P_2, \ldots, P_{v-1}$, and replace each entry $x$ by $x-v+1$ for all $x \geq v$.
In this way, we get a tableau $U$ of shape $D_{\alpha'/\beta'}$ where $\alpha' = [a-v+1,b-v+1]$ and $\beta' = [(a-v+1)-1,1]$ such that $U_{(j-v+1)} \cap U^{(1)}$ is not a $(j-v+1,b-v+1)$-hook at $(1,1)$; a contradiction to the proven fact that $T_{(j)} \cap T^{(1)}$ is a $(j,b)$-hook at $(1,1)$ for each $1 \leq j \leq \min\{a,b\}$ if $T$ is of shape $D_{[a,b]/[a-1,1]}$.\par
Case 2: $w < a-1$.\par
The tableau $T_{(w+1)}$ is a tableau of shape $D_{\alpha'/\beta'}$ where $\alpha' = [w+1,a+b-w-1]$ and $\beta' = [w,1]$.
Thus, $P_1$ is a $(a+b-w-1,b)$-hook at $(1,1)$.
After removing $P_1$ and replacing each entry $x$ by $x-1$ and $x'$ by $(x-1)'$ for all $2 \leq x \leq \ell(c(T))$, we get a tableau of shape $D_{\alpha''/\beta''}$ where $\alpha'' = [a-1,b]$ and $\beta'' = [w,1]$ where $w \leq a-2$.
Using the same argument, $P_2$ is a $(w+1,a+b-w-2)$-hook at $(2,2)$.\par
Repeating this argument, we find that all non-empty $P_i$s are hooks at $(i,i)$ and, therefore, the filling of the boxes of the first $k$ rows of $D$ in any amenable tableau $T$ is unique up to marks.
\end{proof}
\begin{Rem}
Since, by the remark after Definition \ref{amenable}, each hook $T^{(i)}$ must be fitting, this shows that there is only one amenable filling for diagrams of shape $D_{\lambda/\mu}$ where $\lambda = [a,b]$ and $\mu = [w,1]$.
Different proofs of this fact were given by Salmasian \cite[Proposition 3.29]{Salmasian} and DeWitt \cite[Theorem IV.3]{DeWitt}.
\end{Rem}

\begin{Lem}\label{cased3}
Let $\lambda = [a,b,1,d]$ and $\mu = [w,1]$.
Then $Q_{\lambda/\mu}$ is $Q$-multiplicity-free.
\end{Lem}
\begin{proof}
Let the diagram $D = D_{\lambda/\mu}$ be shifted such that the uppermost leftmost box is $(1,1)$.
Since case $w = 1$ is shown in Lemma \ref{casea}, we only have to show case $w \geq 2$.
The subdiagram consisting of the first $a$ rows is $D_{\alpha/\beta}$ where $\alpha = [a,p]$ and $\beta = [q,1]$ for some $p,q$.
By Lemma \ref{uniquefilling}, it has a unique filling up to marks in the $a^{\textrm{th}}$ row.\par
Suppose there are two amenable tableaux $T_1$ and $T_2$ of $D$ of the same content.
Then the difference between these two tableaux are marks since the content of the $(a+1)^{\textrm{th}}$ row and, therefore, the filling of this row up to marks is determined.
Thus, there is a minimal $k$ such that an entry $k$ is in the lowermost row and there is a box $(a,k)$ with entry $k'$ in $T_1$, say, and with entry $k$ in $T_2$.
Since the $k$ in the $(a+1)^{\textrm{th}}$ row must be in a column to the left of the $k^{\textrm{th}}$ column, we have $k > 1$.
In $T_2$, if there is no $k-1$ in the $(a+1)^{\textrm{th}}$ row then $c^{(u)}(T_2)_k = b = c^{(u)}(T_2)_{k-1}$, which is a contradiction to Lemma \ref{unmarked}.
Thus, there is a $k-1$ in the $(a+1)^{\textrm{th}}$ row in a box to the left of the $(k-1)^{\textrm{th}}$ column.
If there is no $k-2$ in the $(a+1)^{\textrm{th}}$ row then $c^{(u)}(T_2)_{k-1} = b = c^{(u)}(T_2)_{k-2}$, which is a contradiction to Lemma \ref{unmarked}.
Thus, there is a $k-2$ in the $(a+1)^{\textrm{th}}$ row in a box to the left of the $(k-2)^{\textrm{th}}$ column.\par
Repeating this argument for $k-3, k-4, \ldots 1$, there must be a $1$ in a box to the left of the first column; a contradiction.
Thus, there are no two amenable tableaux $T_1$ and $T_2$ of $D$ of the same content.
\end{proof}
\begin{Ex}
For $\lambda = [4,2,1,3]$ and $\mu = [3,1]$ we have
$Q_{(9,8,7,6,3)/(3,2,1)} = Q_{(9,8,6,4)}+Q_{(9,8,6,3,1)}+Q_{(9,8,5,4,1)}+Q_{(9,8,5,3,2)}+Q_{(9,7,6,4,1)}+Q_{(9,7,6,3,2)}+Q_{(9,7,5,4,2)}.$
\end{Ex}
\begin{Co}\label{cased1}
Let $\lambda = [1,b,c,d]$ and $\mu = [w,1]$.
Then $Q_{\lambda/\mu}$ is $Q$-multiplicity-free.
\end{Co}
\begin{proof}
For each tableau $T$ of shape $D_{\lambda/\mu}$ let $R_T$ be the diagram of the tableau after removing the boxes of $T^{(1)}$.
By Lemma \ref{firstrows}, the first row has only entries from $\{1', 1\}$.
Two amenable tableaux $T_1$ and $T_2$ of shape $D_{\lambda/\mu}$ such that $R_{T_1} \neq R_{T_2}$ cannot have the same content because then $c(T_1)_1 \neq c(T_2)_1$.
Thus, $R_T = R_{T_1} = R_{T_2}$ has shape $D_{\alpha/\beta}$ where $\alpha = [c,y]$ and $\beta \in \{[v,1], [v,2], [z,1,v,1]\}$ for some $v$ and $z$.
If for all $T$ the diagram $R_T$ has no two amenable tableaux of the same content then $Q_{\lambda/\mu}$ is $Q$-multiplicity-free.\par
We have $R_T^{ot} = D_{\alpha'/\beta'}$ where $\alpha' = [c+y-v-1,v+1]$ and $\beta' = [y-1,1]$ for $\alpha = [c,y]$ and $\beta = [v,1]$.
We have $R_T^{ot} = D_{\alpha'/\beta'}$ where $\alpha' = [c+y-v-2,v,1,1]$ and $\beta' = [y-1,1]$ for $\alpha = [c,y]$ and $\beta = [v,2]$.
In addition, we have $R_T^{ot} = D_{\alpha'/\beta'}$ where $\alpha' = [c+y-z-v-2,z,1,v+1]$ and $\beta' = [y-1,1]$ for $\alpha = [c,y]$ and $\beta = [z,1,v,1]$.\par
By Lemmas \ref{uniquefilling} and \ref{cased3}, in each of these cases $R_T^{ot}$ does not have two amenable tableaux of the same content.
Thus, $Q_{\lambda/\mu}$ is $Q$-multiplicity-free.
\end{proof}
\begin{Ex}
For $\lambda = [1,4,5,2]$ and $\mu = [3,1]$ we have
$Q_{(11,6,5,4,3,2)/(3,2,1)} = Q_{(11,6,5,3)}+Q_{(10,6,5,4)}+Q_{(10,6,5,3,1)}+Q_{(9,6,5,4,1)}+Q_{(9,6,5,3,2)}+Q_{(8,6,5,4,2)}.$
\end{Ex}

\begin{Lem}\label{cased2}
Let $\lambda = [a,1,c,d]$, $d \neq 1$ and $\mu = [w,1]$.
Then $Q_{\lambda/\mu}$ is $Q$-multiplicity-free.
\end{Lem}
\begin{proof}
Consider $D_{\lambda/\mu}^{ot} = D_{\lambda'/\mu'}$ where $\lambda' = [a+c+d-w,w+1]$ and $\mu' = [1,c,d-1,1]$.
Thus, we have $\lambda' = (a+c+d,a+c+d-1,\ldots,w+1)$ and $\mu' = (c+d,d-1,d-2,\ldots,1)$.\par
Since $f^{\lambda'}_{\mu' \nu} = f^{\lambda'}_{\nu \mu'}$, we need to look at tableaux of shape $D_{\lambda'/\nu}$ and content $\mu'$.
See Example \ref{cased2ex} for a depiction of the proof.\par
Let $T$ and $T'$ be two different amenable tableaux of shape $D_{\lambda'/\nu}$ and content $\mu'$.
By Lemma \ref{unmarked}, all $2, 3, \ldots, d = \ell(\mu')$ are unmarked.
Since $d$ is the largest entry, it must be in a corner.
Since there is only one corner, say $(x,y)$, we have $T(x,y) = T'(x,y) = d$.
Next insert the $(d-1)$s.
Both $(d-1)$s must be unmarked and at least one $d-1$ must be in the $y^{\textrm{th}}$ column, otherwise the tableau is not amenable.
Thus, we have $T(x-1,y) = T'(x-1,y) = d-1$ and the other $d-1$ is in the lowermost box in the $(y-1)^{\textrm{th}}$ column.
Repeating this argument, we see that the numbers $2, 3, \ldots, d$ are distributed as follows:
For $0 \leq i \leq d-2$ in the $(y-i)^{\textrm{th}}$ column the lowermost boxes are filled from bottom to top with $d-i, d-i-1, \ldots, 2$.
This is fixed for all amenable tableaux of the given shape.
To get an amenable tableau there must be an unmarked $1$ in each column with a $2$ and in at least one column with no $2$.\par
If there are two amenable tableaux of the same shape then they differ only by markings on some $1$s.
Let $(u,v)$ be such that $T(u,v) = 1'$ and $T'(u,v) = 1$ or vice versa.
Then $T(u+1,v), T'(u+1,v), T(u,v-1), T'(u,v-1) \notin \{1', 1\}$.
Thus, either $(u,v)$ is in the lowermost row of the $v^{\textrm{th}}$ column or $T(u+1,v) = T'(u+1,v) = 2$.
If $T(u+1,v) = T'(u+1,v) = 2$ then $T(u,v) = T'(u,v) = 1$ as mentioned above.
By the remark after Definition \ref{amenable}, the leftmost box of the lowermost row with boxes that are filled with entry from $\{1', 1\}$ must contain a $1$.
Thus, there is no such box $(u,v)$ and, therefore, there are no two amenable tableaux of the same shape.
\end{proof}
\begin{Ex}\label{cased2ex}
Let $\lambda = (12,11,10,8,7,6,5,4)$ and $\mu = (4,3,2,1)$.
Then we have
$${\Yvcentermath1 D_{\lambda/\mu} = \young(\none \none \none \none \none \none \none \none ,\none \none \none \none \none \none \none \none ,\none \none \none \none \none \none \none \none ,\none \none \none \none \none \none \none ,\none \none \none \none \none \none \none ,:\none \none \none \none \none \none ,::\none \none \none \none \none ,:::\none \none \none \none )}.$$
Since $f^{\lambda}_{\mu \nu} = f^{\lambda'}_{\mu' \nu}$ and $D_{\lambda/\mu}^{ot} = D_{\lambda'/\mu'}$ where $\lambda'/\mu' = (12,11,10,9,8,7,6,5)/(9,3,2,1)$ we can consider $D_{\lambda'/\mu'}$:
$${\Yvcentermath1 D_{\lambda'/\mu'} = \young(:::::\none \none \none ,\none \none \none \none \none \none \none \none ,\none \none \none \none \none \none \none \none ,\none \none \none \none \none \none \none \none ,\none \none \none \none \none \none \none \none ,:\none \none \none \none \none \none \none ,::\none \none \none \none \none \none ,:::\none \none \none \none \none )}.$$
Since $f^{\lambda'}_{\mu' \nu} = f^{\lambda'}_{\nu \mu'}$ we can consider amenable tableaux of shape $D_{(12,11,10,9,8,7,6,5)/\nu}$ and content $(9,3,2,1)$.
We know fixed entries:
$${\Yvcentermath1 \tilde{T} = \young(\none \none \none \none \none \none \none \none \none \none \none \none ,:\none \none \none \none \none \none \none \none \none \none \none ,::\none \none \none \none \none \none \none \none \none \none ,:::\none \none \none \none \none \none \none \none \none ,::::\none \none \none \none \none \none \none 1,:::::\none \none \none \none \none 12,::::::\none \none \none 123,:::::::\none 1234)}.$$
Now we have five entries from $\{1', 1\}$ left to put into boxes such that we get an amenable tableau.
For example we obtain
$${\Yvcentermath1 T = \young(\none \none \none \none \none \none \none \none \none \none \none \none ,:\none \none \none \none \none \none \none \none \none \none \none ,::\none \none \none \none \none \none \none \none \none \meins ,:::\none \none \none \none \none \none \none \none \meins ,::::\none \none \none \none \none \none \meins 1,:::::\none \none \none \none \none 12,::::::\none \none \meins 123,:::::::11234)},$$
which is the only tableau of shape $D_{(12,11,10,9,8,7,6,5)/(12,11,9,8,6,5,2)}$ and content $(9,3,2,1)$.
Thus, we have
$$f^{(12,11,10,8,7,6,5,4)}_{(4,3,2,1) (12,11,9,8,6,5,2)} = f^{(12,11,10,9,8,7,6,5)}_{(9,3,2,1) (12,11,9,8,6,5,2)} = f^{(12,11,10,9,8,7,6,5)}_{(12,11,9,8,6,5,2) (9,3,2,1)} = 1.$$
\end{Ex}

\begin{Lem}\label{cased4}
Let $\lambda = [a,b,c,d]$ and $\mu = [w,1]$ where $w \leq 2$.
Then $Q_{\lambda/\mu}$ is $Q$-multiplicity-free.
\end{Lem}
\begin{proof}
Case $w = 1$ follows from Lemma \ref{casea}.
Thus, consider case $w = 2$.
Since $f^{\lambda}_{\mu \nu} = f^{\lambda}_{\nu \mu}$, we may consider tableaux of shape $D_{\lambda/\nu}$ and content $(2,1)$.
There are two words with content $(2,1)$, namely $w^{(1)} = 121$ and $w^{(2)} = 211$.
If $Q_{\lambda/\mu}$ is not $Q$-multiplicity-free then there must be some $\nu$ such that $D_{\lambda/\nu}$ is a diagram with two tableaux $T_1$ and $T_2$ where $c(T_1) = w^{(1)}$ and $c(T_2) = w^{(2)}$.
If $(x(2),y(2)) = (x(3),y(3)-1)$ then $T_1(x(2),y(2)) = 2$ and $T_1(x(3),y(3)) = 1$ and $T_1$ is not a tableau; a contradiction.
If $(x(2),y(2)) = (x(3)+1,y(3))$ then $T_2(x(2),y(2)) = 1$ and $T_2(x(3),y(3)) = 1$ and $T_2$ is not a tableau; a contradiction.
Similarly, we have $(x(1),y(1)) \neq (x(2),y(2)-1)$ and $(x(1),y(1)) \neq (x(2)+1,y(2))$.
Thus, these three boxes are all in different components consisting of one box.
Each component of a diagram has a corner, hence, $\lambda$ has at least three corners; a contradiction to $\lambda = [a,b,c,d]$.
\end{proof}

Lemma \ref{cased3}, Corollary \ref{cased1}, Lemma \ref{cased2} and Lemma \ref{cased4} together prove that \ref{mf1} (\ref{d}) is $Q$-multiplicity-free.

\begin{Lem}\label{casec}
Let $\lambda = [a,b]$ and $\mu = [w,x,y,1]$ where $w = 1$ or $x = 1$ or $2 \leq b \leq 3$ or $a+b-w-x-y-1 = 1$.
Then $Q_{\lambda/\mu}$ is $Q$-multiplicity-free.
\end{Lem}
\begin{proof}
Let $D = D_{\lambda/\mu}$, where $\lambda = [a,b]$ and $\mu = [w,x,y,1]$.
Then $D^{ot}$ has shape $D_{\alpha/\beta}$ where $\alpha=[a+b-w-x-y-1,w,x,y+1]$ and $\beta=[b-1,1]$.
For each of the given restrictions we have one of the following cases.\par
Case $w = 1$: Then we have $\alpha = [a+b-x-y-2,1,x,y+1]$ and Lemma \ref{cased2} proves $Q$-multiplicity-freeness.\par
Case $x = 1$: Then we have $\alpha = [a+b-w-y-2,w,1,y+1]$ and Lemma \ref{cased3} proves $Q$-multiplicity-freeness.\par
Case $2 \leq b \leq 3$: Then we have $\beta=[z,1]$ where $1 \leq z \leq 2$ and Lemma \ref{cased4} proves $Q$-multiplicity-freeness.\par
Case $a+b-w-x-y-1 = 1$: Then we have $\alpha = [1,w,x,y+1]$ and Corollary \ref{cased1} proves $Q$-multiplicity-freeness.
\end{proof}

Thus, we have shown that \ref{mf1} (\ref{c}) is $Q$-multiplicity-free by showing that \ref{mf1} (\ref{c}) is the orthogonal transpose of \ref{mf1} (\ref{d}).
Now we will prove the $Q$-multiplicity-freeness of \ref{mf1} (\ref{f}) and afterwards we will show that the orthogonal transpose of \ref{mf1} (\ref{e}) is included in \ref{mf1} (\ref{f}) which means that the last remaining case of Proposition \ref{mf1} is proved to be $Q$-multiplicity-free.

\begin{Lem}\label{casee1}
Let $\lambda = [a,b,c,1]$ and $\mu = [w,1]$ where $a \leq 2$.
Then $Q_{\lambda/\mu}$ is $Q$-multiplicity-free.
\end{Lem}
\begin{proof}
Since case $a = 1$ is shown in Corollary \ref{cased1}, we only have to show case $a = 2$.
For each tableau $T$ of shape $D_{\lambda/\mu}$ let $R_T$ be the diagram of the remaining tableau after removing the boxes with entry from $\{1', 1, 2', 2\}$.
By Lemma \ref{firstrows}, the first two rows only have entries from $\{1', 1, 2', 2\}$.
The boxes with entry from $\{1', 1\}$ form a hook.
If the boxes with entry from $\{2', 2\}$ form a border strip all the marks of the entries are determined.
If the boxes with entry from $\{2', 2\}$ form a diagram with more than one component then it must have precisely two components.
The first component has boxes only in the $(w+1)^{\textrm{th}}$ column and the second component has boxes in all other columns.
In this case the last box of the second component must contain a $2'$ by the remark after Definition \ref{amenable} and by Lemma \ref{unmarked}.
Thus, there are no two tableaux differing just by marks on the entries from $\{1', 1, 2', 2\}$.\par
If no $R_T$ for any $T$ has two amenable tableaux of the same content then $Q_{\lambda/\mu}$ is $Q$-multiplicity-free.
$R_T^{ot}$ is a diagram of shape $D_{\alpha'}$ for some $\alpha' \in DP$.
Such a diagram has only one amenable tableau, namely the one that has just $i$s in the $i^{\textrm{th}}$ row for $1 \leq i \leq \ell(\alpha')$.
Thus, $Q_{\lambda/\mu}$ is $Q$-multiplicity-free.
\end{proof}
\begin{Ex}
For $\lambda = [1,5,6,1]$ and $\mu = [4,1]$ we have
$$Q_{(12,6,5,4,3,2,1)/(4,3,2,1)} = Q_{(12,6,5)}+Q_{(11,6,5,1)}+Q_{(10,6,5,2)}+Q_{(9,6,5,3)}+Q_{(8,6,5,4)}.$$
For $\lambda = [2,5,5,1]$ and $\mu = [4,1]$ we have\par
$Q_{(12,11,5,4,3,2,1)/(4,3,2,1)} = Q_{(12,11,5)}+Q_{(11,10,5,2)}+Q_{(11,9,5,3)}+Q_{(11,9,5,2,1)}+Q_{(11,8,5,4)}+Q_{(11,8,5,3,1)}+Q_{(11,7,5,4,1)}+Q_{(10,9,5,3,1)}+Q_{(10,8,5,4,1)}+Q_{(10,8,5,3,2)}+Q_{(10,7,5,4,2)}+Q_{(9,8,5,4,2)}+Q_{(9,7,5,4,3)}+Q_{(12,10,5,1)}+Q_{(12,9,5,2)}+Q_{(12,8,5,3)}+Q_{(12,7,5,4)}$.
\end{Ex}

\begin{Lem}\label{casee2}
Let $\lambda = [a,b,c,1]$ and $\mu = [w,1]$ where $b \leq 2$.
Then $Q_{\lambda/\mu}$ is $Q$-multiplicity-free.
\end{Lem}
\begin{proof}
Case 1: $b = 1$.\par
The diagram $D_{\lambda/\mu}^{ot}$ has shape $D_{\alpha/\beta}$ where $\alpha = [a+c+1-w,w+1]$ and $\beta = [1,c+1]$.
Thus, $\alpha = (a+c+1,a+c,\ldots,w+1)$ and $\beta = (c+1)$.
Then $B_{\alpha}$ is a rotated hook and every diagram from $B_{\alpha}^{(n)}$ is connected.
By Proposition \ref{lambda/n}, $Q_{\alpha/\beta} = Q_{\lambda/\mu}$ is $Q$-multiplicity-free.\par
Case 2: $b = 2$.\par
The diagram $D_{\lambda/\mu}^{ot}$ has shape $D_{\alpha/\beta}$ where $\alpha = [a+c-w+2,w+1]$ and $\beta = [2,c+1]$.
Thus, $\alpha = (a+c+2,a+c+1,\ldots,w+1)$ and $\beta = (c+2, c+1)$.
Since $f^{\alpha}_{\beta \nu} = f^{\alpha}_{\nu \beta}$, we need to look at amenable tableaux of shape $D_{\alpha/\nu}$ and content $(c+2, c+1)$.
The boxes with an entry from $\{2', 2\}$ form a border strip (in fact a rotated hook) where marks are determined.
In every column with a box of this border strip there is a box filled with $2$.
To obtain an amenable tableau in each of these columns there must be a box filled with a $1$.
Above the uppermost box filled with a $1$ there cannot be a box filled with a $1'$.
Otherwise, if $w$ is the reading word of this tableau and the uppermost box filled with $1$ is $(x(j),y(j))$ then $c+1 = m_2(\ell(w)+j-1) \geq m_1(\ell(w)+j-1)$ and $w_j = 1$; a contradiction to the amenability of the tableau.\par
Suppose we have two amenable tableaux $T$ and $T'$ of shape $D_{\lambda/\nu}$.
If there are boxes $(x,y)$ such that $T(x,y) \in \{2', 2\}$ and $T'(x,y) \in \{1', 1\}$ then one of these boxes is either the first or the last box of $T^{(2)}$.
But then there is a box $(r,s)$ such that $T(r,s) \in \{1', 1\}$ and $T'(r,s) \in \{2', 2\}$ is the last box or the first box of $T'^{(2)}$, respectively.
Without loss of generality we may assume that $(x,y)$ is the first box of $T^{(2)}$.
Then $T(x-1,y) = 1$ and $(x-2,y)$ is not part of the diagram.
Since $T'(x,y) \in \{1', 1\}$, we have $T'(x-1,y) = 1'$; a contradiction to the fact that there cannot be a box filled with a $1'$ above the uppermost box filled with a $1$.\par
Hence, $T$ and $T'$ differ only by markings on $1$s.
Let $(u,v)$ be the uppermost rightmost box such that $T'(u,v) = 1'$, say, and $T(u,v) = 1$.
Then the boxes $(u+1,v), (u,v-1) \notin T^{(1)}=T'^{(1)}$.
Thus, either $(u+1,v) \notin D_{\lambda/\nu}$ or $T(u+1,v) = T'(u+1,v) \in T^{(2)} = T'^{(2)}$.
Suppose $T(u+1,v) = T'(u+1,v) \in T^{(2)} = T'^{(2)}$.
If we have $(u+1,v) = (x(k),y(k))$ then for $w(T')$ we have $m_2(\ell(w(T'))-k) = m_1(\ell(w(T'))-k)$ and $w_k \in \{2', 2\}$; a contradiction to the amenability of $T'$.
Hence, $(u+1,v) \notin D_{\lambda/\nu}$.
By the remark after Definition \ref{amenable}, $T^{(1)} = T'^{(1)}$ is fitting.
It follows that there is no box $(u,v)$ and, therefore, there are no two amenable tableaux of $D_{\lambda/\nu}$.
\end{proof}
\begin{Ex}
For $\lambda = [3,1,6,1]$ and $\mu = [6,1]$ we have
$$Q_{(10,9,8,6,5,4,3,2,1)/(6,5,4,3,2,1)} = Q_{(10,9,8)}+Q_{(10,9,7,1)}+Q_{(10,8,7,2)}+Q_{(9,8,7,3)}.$$
For $\lambda = [3,2,6,1]$ and $\mu = [6,1]$ we have\par
$Q_{(11,10,9,6,5,4,3,2,1)/(6,5,4,3,2,1)} = Q_{(11,10,9)}+Q_{(11,10,8,1)}+Q_{(11,10,7,2)}+Q_{(11,9,8,2)}+Q_{(11,9,7,3)}+Q_{(11,9,7,2,1)}+Q_{(11,8,7,3,1)}+Q_{(10,9,8,3)}+Q_{(10,9,7,3,1)}+Q_{(10,9,7,4)}+\linebreak
Q_{(10,8,7,4,1)}+Q_{(10,8,7,3,2)}+Q_{(9,8,7,4,2)}$.
\end{Ex}

\begin{Lem}\label{casee3}
Let $\lambda = [a,b,c,1]$ and $\mu = [w,1]$ where $c \leq 2$.
Then $Q_{\lambda/\mu}$ is $Q$-multiplicity-free.
\end{Lem}
\begin{proof}
Let $n = |D_{\lambda/\mu}|$.\par
Case 1: $c = 1$.\par
The only box in the $(a+1)^{\textrm{th}}$ row is $(a+1,a+1)$.
By Lemma \ref{uniquefilling}, the filling of the first $a$ rows is unique up to markings.
In fact, the filling consists entirely of hooks at the diagonal $\{(s,t) \mid t-s = w\}$.
Thus, two different amenable tableaux of the same content differ only by markings.
Suppose we have two such tableaux $T$ and $T'$.
Let $(y,z)$ be a box such that $T'(y,z) = k'$, say, and $T(y,z) = k$.
Then there must be a box below and to the left of this box with a $k$.
This box is $(a+1,a+1)$ and $y = a$.
However, since $T(a,z) = k$, we have $m_{k-1}(n) = m_k(n)$; a contradiction to Lemma \ref{unmarked}.
Thus, there are no two different amenable tableaux of the same content.\par
Case 2: $c = 2$.\par
Let $T$ be an amenable tableau of shape $D_{\lambda/\mu}$.
By Lemma \ref{uniquefilling}, the filling of the first $a$ rows is unique up to markings.
In fact, the filling consists entirely of hooks at the diagonal $\{(s,t) \mid t-s = w\}$.
The three boxes below the $a^{\textrm{th}}$ row are $(a+1,a+1)$, $(a+1,a+2)$ and $(a+2,a+2)$.\par
Case 2.1: $|T(a+1,a+1)| = |T(a+1,a+2)| = k$ for some $k$.\par
Then, by Lemma \ref{diagonal}, we have $|T(a+2,a+2)| > k$.
Since $(a,a+1) \in D_{\lambda/\mu}$ we have $k > 1$.
If $k'$ or $k$ occur in the first $a$ rows, it follows that $m_k(2n) \geq m_{k-1}(2n)$; a contradiction to the amenability of $T$.
Thus, $k = j+1$, where $j = \min\{a, b+3\}$.
This is only possible if there are at least three unmarked $j$s, otherwise there is no amenable tableau with these properties.
Then $T(a+2,a+2) = k+1 = j+2$ follows and $T(a+1,a+1)$, $T(a+1,a+2)$ and $T(a+2,a+2)$ are unmarked.
Additionally, each of the entries in the $a^{\textrm{th}}$ row is unmarked and, therefore, there is no other amenable tableau of the same content.\par
Case 2.2: $|T(a+1,a+2)| = |T(a+2,a+2)| = k$ for some $k$.\par
Since $(a,a+1) \in D_{\lambda/\mu}$ we have $k > 1$.
If $k'$ or $k$ occur in the first $a$ rows it follows that $T(a+1,a+1) = k-1$, otherwise $m_k(2n) \geq m_{k-1}(2n)$; a contradiction to the amenability of $T$.
Assume there are two different amenable tableaux $T$ and $T'$ of $D_{\lambda/\mu}$ of the same content such that $|T(a+1,a+1)| = |T'(a+1,a+1)| = k-1$, $|T(a+1,a+2)| = |T'(a+1,a+2)| = k$ and $|T(a+2,a+2)| = |T'(a+2,a+2)| = k$.
It follows that these tableaux differ only by markings.
Then there is some $i$ such that $T'(y,z) = i'$, say, and $T(y,z) = i$.
It follows that $y = a$ since the entries in the other rows are determined.
It also follows that there is an $i$ in a box which is lower and to the left of $(a,z)$.
Thus, we have $i \in \{k-1, k\}$ and, therefore, $k > 2$.
If $i = k-1$ then, since $T(a,z) = k-1$, for $w(T)$ we have $m_{k-2}(n) = m_{k-1}(n)$; a contradiction to Lemma \ref{unmarked}.
Hence, we have $i = k$.
If $T(a,z-1) = (k-1)'$, then, since $T(a,z) = k$, for $w(T')$ we have $m_{k-1}(n) = m_k(n)$; again a contradiction to Lemma \ref{unmarked}.
If $T(a,z-1) = k-1$, then we have $m_{k-2}(n) = m_{k-1}(n)$; a contradiction to Lemma \ref{unmarked} as well.
Thus, there are no such two different amenable tableaux of $D_{\lambda/\mu}$.\par
Case 2.3: $|T(a+1,a+1)| = u$, $|T(a+1,a+2)| = v$ and $|T(a+2,a+2)| = t$ where $u \neq v$, $u \neq t$ and $v \neq t$.\par
Then we have $u < v < t$.
Assume there are two different amenable tableaux $T$ and $T'$ of $D_{\lambda/\mu}$ of the same content in which the boxes $(a+1,a+1)$, $(a+1,a+2)$ and $(a+2,a+2)$ are filled as above.
It follows that these tableaux differ only by markings.
Then there is some $i$ such that $T'(y,z) = i'$, say, and $T(y,z) = i$.
It follows that $y = a$ since the entries in the other rows are determined.
It also follows that there is an $i$ in a box which is lower and to the left of the box $(a,z)$.
The only possible case is that $i \in \{u, v, t\}$.
Arguing as in the cases above, we see that for $T$ we either have $m_{t-1}(n) = m_t(n)$ or $m_{v-1}(n) = m_v(n)$ or $m_{u-1}(n) = m_u(n)$.
This contradicts Lemma \ref{unmarked}.\par
Hence, there are no such two different amenable tableaux of $D_{\lambda/\mu}$.
\end{proof}
\begin{Ex}
For $\lambda = [5,3,1,1]$ and $\mu = [4,1]$ we get
$$Q_{(9,8,7,6,5,1)/(4,3,2,1)} = Q_{(9,8,5,3,1)}+Q_{(9,7,6,3,1)}+Q_{(9,7,5,4,1)}+Q_{(9,7,5,3,2)}.$$
For $\lambda = [4,3,2,1]$ and $\mu = [4,1]$ we get\par
$Q_{(9,8,7,6,2,1)/(4,3,2,1)} = Q_{(9,7,5,2)}+Q_{(9,8,4,2)}+Q_{(8,6,5,4)}+Q_{(8,6,5,3,1)}+Q_{(8,6,4,3,2)}+Q_{(8,7,4,3,1)}+Q_{(8,7,5,2,1)}+Q_{(8,7,6,2)}+Q_{(8,7,5,3)}+Q_{(9,6,4,3,1)}+Q_{(9,6,5,2,1)}+Q_{(9,7,4,3)}+Q_{(9,7,4,2,1)}+Q_{(9,6,5,3)}$.
\end{Ex}

\begin{Lem}\label{casee4}
Let $\lambda = [a,b,c,1]$ and $\mu = [w,1]$ where $w \leq 3$ or $w = a+c-1$.
Then $Q_{\lambda/\mu}$ is $Q$-multiplicity-free.
\end{Lem}
\begin{proof}
Case $w = 1$ follows from Lemma \ref{casea} and case $w = 2$ follows from Lemma \ref{cased4}.
For case $w = a+c-1$ the diagram $D_{\lambda/\mu}^t$ has shape $D_{\alpha/\beta}$ where $\alpha = [1,c,b,a]$ and $\beta = [b,1]$ and follows from Corollary \ref{cased1}.
Thus, we only have to prove case $w = 3$.\par
Since $f^{\lambda}_{\mu \nu} = f^{\lambda}_{\nu \mu}$, we just need to look at tableaux of shape $D_{\lambda/\nu}$ and content $\mu = (3,2,1)$.
By Lemma \ref{unmarked}, all entries must be unmarked.
Assume there are two different amenable tableaux $T_1$, $T_2$ of $D_{\lambda/\nu}$ with content $\mu$ for some $\nu \in DP$.
Thus, we get one tableau from the other by interchanging some entries in certain boxes.\par
Suppose the $3$ is in one of these boxes.
Let $(a,x)$ be the upper corner (where $x = a+b+c$) and let $(e,e)$ be the lower corner (where $e = a+c$).
Since the $3$ is the greatest entry it must be either in $(a,x)$ or in $(e,e)$.
Thus, we have $T_1(a,x) = 3$, say, and $T_2(e,e) = 3$.
Then, by Lemma \ref{firstrows} and since $T_1$ is amenable, we have $a \geq 3$, $T_1(a-1,x) = 2$ and $T_2(a-2,x) = 1$.
We have $T_2(a,x) \in \{1, 2\}$.
Either way, since all entries are unmarked, we have $T_2(a-2,x) \leq T_2(a-1,x)-1 \leq T_2(a,x)-2$ and, hence, $T_2(a-2,x) \notin \{1, 2, 3\}$.
Thus, either $T_1(a,x) = T_2(a,x) = 3$ or $T_1(e,e) = T_2(e,e) = 3$.\par
Suppose $T_1(a,x) = T_2(a,x) = 3$.
Then $T_1(a-1,x) = T_2(a-1,x) = 2$ and $T_1(a-2,x) = T_2(a-1,x) = 1$.
Thus, $T_1$ and $T_2$ differ only by interchanging one $1$ and one $2$.
Let the boxes containing these entries be $(f,t)$ and $(v,g)$, where $g > t$ and $v < f$.
The remaining $1$ must be in a box to the right and above $(v,g)$.
If $T_1(a-1,x-1) = T_2(a-1,x-1) = 1$ then $T_1(a,x-1) = T_2(a,x-1) = 2$ and both tableaux are the same; a contradiction.
Thus, we have $T_1(a,x-1) = T_2(a,x-1) = 1$.
The remaining entries must be in two corners below $(a,x-1)$.
However, there is only one corner (namely $(e,e)$), thus, there are no two different amenable tableaux such that $T_1(a,x) = T_2(a,x) = 3$.
Therefore, we have $T_1(e,e) = T_2(e,e) = 3$.\par
Suppose $T_1(a,x) = 1$.
Then $T_1(e-1,e) = T_1(e-1,e-1) = 2$ and after inserting the $1$s the tableau is determined.
Thus, if $T_1(a,x) = 1$, there are no two different amenable tableaux.\par
Therefore, $T_1(a,x) = T_2(a,x) = 2$.
Since $T_1$ and $T_2$ are amenable, $T_1(a-1,x) = T_2(a-1,x) = 1$.
Thus, $T_1$ and $T_2$ differ only by interchanging one $1$ and one $2$.
With the same argument as above we see that $T_1(a,x-1) = T_2(a,x-1) = 1$.
Then we have $T_1(e-1,e) = T_2(e-1,e) = 2$ and both tableaux are the same; a contradiction.
Thus, there are no two different amenable tableaux of shape $D_{\lambda/\nu}$ and content $\mu = (3,2,1)$.
\end{proof}
\begin{Ex}
For $\lambda = [3,3,3,1]$ and $\mu = [3,1]$ we get\par
$Q_{(9,8,7,3,2,1)/(3,2,1)} = Q_{(9,8,7)}+Q_{(9,8,6,1)}+Q_{(9,8,5,2)}+Q_{(9,8,4,3)}+Q_{(9,7,6,2)}+Q_{(9,7,5,3)}+Q_{(9,7,5,2,1)}+Q_{(9,7,4,3,1)}+Q_{(9,6,5,3,1)}+Q_{(9,6,4,3,2)}+Q_{(8,7,6,3)}+Q_{(8,7,4,3,2)}+Q_{(8,6,5,3,2)}+Q_{(8,6,4,3,2,1)}+Q_{(8,7,5,3,1)}$.
\end{Ex}

The Lemmas \ref{casee1}, \ref{casee2}, \ref{casee3} and \ref{casee4} all together show that \ref{mf1} (\ref{e}) is $Q$-multiplicity-free.

\begin{Lem}\label{casef}
Let $\lambda = [a,b]$ and $\mu = [w,x]$ where $2 \leq b \leq 4$ or $w \leq 2$ or $2 \leq x \leq 3$ or $a = w+1$ or $a+b-w-x \leq 2$.
Then $Q_{\lambda/\mu}$ is $Q$-multiplicity-free.
\end{Lem}
\begin{proof}
The diagram $D_{\lambda/\mu}^{ot}$ has shape $D_{\alpha/\beta}$ where $\alpha=[a+b-w-x,w,x-1,1]$ and $\beta=[b-1,1]$.
For each of the given restrictions we have one of the following cases.\par
Case $2 \leq b \leq 4$: Then we have $\beta = [w',1]$ where $w' \leq 3$ and Lemma \ref{casee4} proves $Q$-multiplicity-freeness.\par
Case $w \leq 2$: Then we have $\alpha = [a',b',c',1]$ where $b' \leq 2$ and Lemma \ref{casee2} proves $Q$-multiplicity-freeness.\par
Case $2 \leq x \leq 3$: Then we have $\alpha = [a',b',c',1]$ where $c' \leq 2$ and Lemma \ref{casee3} proves $Q$-multiplicity-freeness.\par
Case $a = w+1$: Then we have $\alpha = [a',b',c',1]$ and $\beta = [w',1]$ where we have $a' = a+b-w-x = b-x+1$ and, hence, $w' = b-1 = (b-x+1)+(x-1)-1 = a'+c'-1$ and Lemma \ref{casee4} proves $Q$-multiplicity-freeness.\par
Case $a+b-w-x \leq 2$: Then we have $\alpha = [a',b',c',1]$ where $a' \leq 2$ and Lemma \ref{casee1} proves $Q$-multiplicity-freeness.
\end{proof}

We have now proven that all the skew Schur $Q$-functions occurring in Proposition \ref{mf1} are indeed $Q$-multiplicity-free, and hence we are now able to state this result as our final classification theorem.

\begin{Th}\label{listmf}
Let $\lambda, \mu \in DP$ and $a,b,c,d,w,x,y \in \NN$ such that $D_{\lambda/\mu}$ is basic.
$Q_{\lambda/\mu}$ is $Q$-multiplicity-free if and only if $\lambda$ and $\mu$ satisfy one of the following conditions:
\begin{enumerate}[(i)]
	\item $\lambda$ is arbitrary and $\mu \in \{\emptyset, (1)\}$,\label{a1}
	\item $\lambda = (a+b-1, a+b-2, \ldots, b)$ where $b \in \{1,2\}$ and $\mu$ is arbitrary,\label{b1}
	\item $\lambda = (a+b-1, a+b-2, \ldots, b)$ and $\mu = (w+x+y, w+x+y-1, \ldots, x+y+2,\linebreak
x+y+1, y, y-1, \ldots, 1)$ where $w = 1$ or $x = 1$ or $b \leq 3$ or $a+b-w-x-y-1 = 1$,\label{c1}
	\item $\lambda = (a+b+c+d-1, a+b+c+d-2, \ldots, b+c+d+1, b+c+d, c+d-1,\linebreak
	c+d-2, \ldots, d)$ where $d \neq 1$ and $\mu = (w, w-1, \ldots, 1)$ where $1 \in \{a,b,c\}$ or $w \leq 2$,\label{d1}
	\item $\lambda = (a+b+c, a+b+c-1, \ldots, b+c+2, b+c+1, c, c-1, \ldots, 1)$ and $\mu = (w, w-1, \ldots, 1)$ where $a \leq 2$ or $b \leq 2$ or $c \leq 2$ or $w \leq 3$ or $w = a+c-1$,\label{e1}
	\item $\lambda = (a+b-1, a+b-2, \ldots, b)$ and $\mu = (w+x-1, w+x-2, \ldots, x)$ where $2 \leq b \leq 4$ or $w \leq 2$ or $x \leq 3$ or $a = w+1$ or $a+b-w-x \leq 2$.\label{f1}
\end{enumerate}
Some of these cases overlap.
\end{Th}
\begin{proof}
Using the shape path notation of Definition \ref{shapepath} we have:
\begin{itemize}
	\item \ref{listmf} (\ref{b1}) is the case $\lambda = [a,b]$ where $b \in \{1,2\}$ and $\mu$ is arbitrary.
	\item \ref{listmf} (\ref{c1}) is the case $\lambda = [a,b]$ and $\mu = [w,x,y,1]$ where $w = 1$ or $x = 1$ or $b \leq 3$ or $a+b-w-x-y-1 = 1$.
	\item \ref{listmf} (\ref{d1}) is the case	$\lambda = [a,b,c,d]$ such that $d \neq 1$ and $\mu = [w,1]$ where $1 \in \{a,b,c\}$ or $w \leq 2$.
	\item \ref{listmf} (\ref{e1}) is the case	$\lambda = [a,b,c,1]$ and $\mu = [w,1]$ where $a \leq 2$ or $b \leq 2$ or $c \leq 2$ or $w \leq 3$ or $w = a+c-1$.
	\item \ref{listmf} (\ref{f1}) is the case	$\lambda = [a,b]$ and $\mu = [w,x]$ where $2 \leq b \leq 4$ or $w \leq 2$ or $x \leq 3$ or $a = w+1$ or $a+b-w-x \leq 2$.
\end{itemize}
By Proposition \ref{mf1}, only these cases can be $Q$-multiplicity-free.
Lemma \ref{casea} states that \ref{listmf} (\ref{a1}) is $Q$-multiplicity-free.
Lemma \ref{caseb} states that \ref{listmf} (\ref{b1}) is $Q$-multiplicity-free.
Lemmas \ref{cased3}, \ref{cased2} and \ref{cased4} and Corollary \ref{cased1} state that \ref{listmf} (\ref{d1}) is $Q$-multiplicity-free.
Lemma \ref{casec} states that \ref{listmf} (\ref{c1}) is $Q$-multiplicity-free.
Lemmas \ref{casee1}, \ref{casee2}, \ref{casee3} and \ref{casee4} state that \ref{listmf} (\ref{e1}) is $Q$-multiplicity-free.
Lemma \ref{casef} states that \ref{listmf} (\ref{f1}) for $x \neq 1$ is $Q$-multiplicity-free.
Lemma \ref{uniquefilling} states that for \ref{listmf} (\ref{f1}) for $x = 1$ we have $Q_{\lambda/\mu} = Q_{\alpha}$ for some $\alpha$ (see the remark after Lemma \ref{uniquefilling}).
Hence, \ref{listmf} (\ref{f1}) for $x = 1$ is $Q$-multiplicity-free.
Thus, all cases in Theorem \ref{listmf} are $Q$-multiplicity-free.
\end{proof}

\begin{Ack}
The QF package for Maple made by John Stembridge (\url{http://www.math.lsa.umich.edu/~jrs/maple.html}) was a helpful tool for analysing the decomposition of skew Schur $Q$-functions.\par
This paper is based on the research I did for my PhD thesis which was supervised by Prof. Christine Bessenrodt.
I am very grateful to Christine Bessenrodt that she introduced me to (skew) Schur $Q$-functions, and I would like to thank her for her help in writing this paper as well as for supervising my research.
\end{Ack}

\end{document}